\documentclass[12pt]{article}
\usepackage{amsmath}
\usepackage{amssymb}
\usepackage{amsthm}
\usepackage[top=20truemm, left=20truemm, right=20truemm]{geometry}

\usepackage{amscd}
\newtheorem{dfn}{Definition}[section]
\newtheorem{thm}[dfn]{Theorem}
\newtheorem{lem}[dfn]{Lemma}
\newtheorem{rem}[dfn]{Remark}
\newtheorem{cor}[dfn]{Corollary}
\newtheorem{prob}[dfn]{Problem}
\newtheorem{prop}[dfn]{Proposition}

\usepackage{amscd}
\usepackage{graphicx}
\usepackage[all]{xy}

\title{A topological invariant for continuous fields of the Cuntz algebras}
\author{Taro Sogabe \\\small Graduate School of Science, Kyoto University, Japan\\
\small staro@math.kyoto-u.ac.jp}

\begin{document}
\maketitle
\abstract
For a continuous field of the Cuntz algebra over a finite CW complex,
we introduce a topological invariant, which is an element in Dadarlat--Pennig's generalized cohomology group,
and prove that the invariant is trivial if and only if the field comes from a vector bundle via Pimsner's construction.

\section{Introduction}
This paper is a continuation of our work \cite{r2},
and our main interest is continuous fields of the Cuntz algebra $\mathcal{O}_{n+1}$.
By M. Dadarlat's result \cite{D2},
continuous fields of $\mathcal{O}_{n+1}$ over a finite CW complex $X$ are automatically locally trivial,
and the homotopy theory is useful to investigate them.
He studied continuous fields of $\mathcal{O}_{n+1}$ over $X$ coming from vector bundles in \cite{D1},
and proved that the ideals of $K$-theory ring defined by the vector bundles give the complete invariant of these fields.
He also showed that every continuous field comes from a vector bundle when the cohomology $H^*(X)$ has no $n$-torsion (see \cite{D1, r2}).
This classification result is used in the work \cite{IM} of M. Izumi and H. Matui for the problem of classifying group actions on $\mathcal{O}_{n+1}$.
However, in general there exist examples of continuous fields which do not come from vector bundles.
The purpose of this paper is to introduce a topological invariant for continuous fields of $\mathcal{O}_{n+1}$,
and we show that the invariant is trivial if and only if the continuous fields come from vector bundles.

To sketch our construction of the invariant,
let us review relevant results.
The Dixmier--Douady theory classifies the locally trivial bundles over a finite CW complex $X$, whose fiber is the algebra of compact operators $\mathbb{K}$, by the third cohomology group $H^3(X)$ (see \cite{RW, CS}).
In their remarkable work, M. Dadarlat and U. Pennig generalized this theory to $\mathbb{K}\otimes D$ for every strongly self-absorbing C*-algebra $D$.
The strongly self-absorbing C*-algebras introduced in \cite{TW} are an important class of C*-algebras containing the Cuntz algebra $\mathcal{O}_2, \mathcal{O}_\infty$ and the Jang--Su algebra $\mathcal{Z}$.
They revealed that the set of the isomorphism classes of the locally trivial continuous fields of $\mathbb{K}\otimes D$ is identified with the first group $E^1_D$ of a generalized cohomology $E^*_D$ (see \cite{DP}).

In this paper, using the Cuntz--Toeplitz algebra $E_{n+1}$, we construct the invariant mentioned above as an element of the group $E^1_D(X)$.
Recall that $E_{n+1}\otimes D$ has a unique proper ideal $\mathbb{K}\otimes D$ and the quotient algebra $(E_{n+1}\otimes D)/(\mathbb{K}\otimes D)$ is isomorphic to $\mathcal{O}_{n+1}\otimes D$.
Our main technical result is as follows :
the group homomorphism $q : \operatorname{Aut}(E_{n+1}\otimes D)\to\operatorname{Aut}(\mathcal{O}_{n+1}\otimes D)$ is a weak homotopy equivalence for every strongly self-absorbing C*-algebra $D$ satisfying the UCT (Corollary \ref{mt3}).
For the proof,
we use the arguments developed in \cite{ST} for computing the homotopy groups of $\operatorname{Aut}(E_{n+1})$.
Then the map $q$, the group homomorphism $\otimes{\rm id}_D : \operatorname{Aut}(\mathcal{O}_{n+1})\to\operatorname{Aut}(\mathcal{O}_{n+1}\otimes D)$ and the restriction map $\eta_n : \operatorname{Aut}(E_{n+1}\otimes D)\to\operatorname{Aut}(\mathbb{K}\otimes D)$ give the invariant $\mathfrak{b}_D : [X, \operatorname{BAut}(\mathcal{O}_{n+1})]\to E^1_D(X)$.
We show that the set $\mathfrak{b}_{M_{(n)}}^{-1}(0)$ consists of the continuous fields coming from vector bundles which are classified in \cite{D1, r2}.
Here we denote by $M_{(n)}$ the UHF algebra of infinite type whose $K_0$-group is the localization of $\mathbb{Z}$ at the ideal $n\mathbb{Z}$ (Theorem \ref{mt4}).

It would be desirable to completely determine the image and the inverse images of the map $\mathfrak{b}_{M_{(n)}}$ but we have not been able to do this.
It is not known whether the image of $\mathfrak{b}_{M_{(n)}}$ is a group or not.
However it is worth pointing out the following two cases.
In the case of ${\rm dim}X\leq 3$,
the classical obstruction theory tells that $[X, \operatorname{BAut}(\mathcal{O}_{n+1})]$ is identified with $H^2(X, \mathbb{Z}_n)$, and $\mathfrak{b}_{M_{(n)}}$ is identified with the Bockstein map $H^2(X, \mathbb{Z}_n)\to \operatorname{Tor}(H^3(X), \mathbb{Z}_n)$ (Theorem \ref{three}).
For the suspension $X=SY$,
the map $\mathfrak{b}_{M_{(n)}}$ reduces to the Bockstein map $\tilde{K}^0(X; \mathbb{Z}_n) \to \operatorname{Tor}(K^1(X), \mathbb{Z}_n)$.
\section*{Acknowledgments}
The author would like to express his greatest appreciation to his supervisor Prof. Masaki Izumi who informed him of Theorem \ref{three} and gave him the idea of the construction of the invariant and many other insightful comments.

\section{Preliminaries}
\subsection{Notation and the basic facts on C*-algebras}

Let $A$ be a unital C*-algebra and let $U(A)$ be the group of unitary elements in $A$.
We denote by $U_0(A)$ the path component of $1_A$ in $U(A)$.
We denote the set of positive elements of $A$ by $A_+$ and denote the set of positive contractions by $(A_+)_1$.
For a non-unital C*-algebra $B$, we denote its unitization by $B^{\sim}$.
The $K$-groups of $A$ are denoted by $K_i(A), i=0, 1$.
We denote by $[p]_0$ the class of a projection $p$ in $K_0(A)$ and by $[u]_1$ the class of a unitary $u$ in $K_1(A)$.
We denote $K_0(A)_+ :=\{ [p]_0\in K_0(A)\; |\; p\in \mathcal{P}(A\otimes \mathbb{K})\}$ where $\mathcal{P}(A\otimes\mathbb{K})$ is the set of projections.
We denote by $SA$ the suspension of a C*-algebra $A$, and one has the Bott periodicity : $K_0(A)\cong K_1(SA)$.
For a compact Hausdorff space $X$,
we denote by $C(X)$ the C*-algebra of continuous functions on $X$ and denote by $C_0(X, x)$ the C*-subalgebra consisting of all functions vanishing at $x\in X$,
which is the kernel of the evaluation map ${\rm ev}_x : C(X)\ni f\mapsto f(x)\in \mathbb{C}$.
We denote $K^i(X)=K_i(C(X))$ and, for connected $X$, denote $\tilde{K}^i(X)=K_i(C_0(X, x))$ for short.
We denote by $\mathbb{M}_n$ the $n$ by $n$ matrix algebra.
For two unital C*-algebras $A_1$ and $A_2$, we denote by $\operatorname{Hom}(A_1, A_2)_u$ the set of unital $*$-homomorphisms from $A_1$ to $A_2$ and denote $\operatorname{End}(A_1):=\operatorname{Hom}(A_1, A_1)_u$.
For a topological space $Y$ and two elements $y_0, y_1\in Y$, we denote $y_0\sim_h y_1$ in $Y$ if there is a continuous path from $y_0$ to $y_1$.
A unital C*-algebra $A$ is called $K_1$-injective if the natural map $U(A)/\sim_h\to K_1(A)$ is injective.

We denote by $\mathbb{K}$ the algebra of compact operators of the infinite dimensional separable Hilbert space $H$.
For an algebra $A\otimes\mathbb{K}$, we denote its multiplier algebra by $\mathcal{M}(A\otimes\mathbb{K})$ and denote by $\mathcal{Q}(A\otimes\mathbb{K})$ the quotient algebra $\mathcal{M}(A\otimes\mathbb{K})/A\otimes\mathbb{K}$ with the quotient map $\pi : \mathcal{M}(A\otimes\mathbb{K})\to \mathcal{Q}(A\otimes \mathbb{K})$.
We remark that $\mathcal{Q}(A\otimes\mathbb{K})$ is $K_1$-injective (see \cite[Section 1.13]{M}).
The algebra $\mathcal{M}(C(X)\otimes A\otimes\mathbb{K})$ is identified with $C_{s}^b(X, \mathcal{M}(A\otimes\mathbb{K}))$, that is, the algebra of $\mathcal{M}(A\otimes\mathbb{K})$-valued bounded continuous functions on $X$ with respect to the strict topology (see \cite[Corollary 3.4]{AP}, \cite[Proposition 2.57]{RW}).

\begin{thm}[{\cite[Theorem 1]{Hig}}]\label{Kui}
For a unital C*-algebra $A$,
the unitary group $U(\mathcal{M}(A\otimes\mathbb{K}))$ is contractible with respect to the norm topology, and we have
$K_i(\mathcal{M}(A\otimes\mathbb{K}))=0,\, i=0, 1.$
\end{thm}

For C*-algebras $B$ and $C$,
an extension $C$ of $B$ by $A\otimes \mathbb{K}$ is an exact sequence
$$0\to A\otimes\mathbb{K}\to C\to B\to 0,$$
and the Busby invariant of the extension is the induced map $\tau : B\to \mathcal{Q}(A\otimes\mathbb{K})$.
We identify an extension with the corresponding Busby invariant.
We refer to \cite{Bl} for the following definition and the basic facts of the theory of extensions of C*-algebras.

\begin{dfn}
Let $A$ be a unital C*-algebra and let $B$ be a C*-algebra.
An extension $\tau : B\to \mathcal{Q}(A\otimes \mathbb{K})$ is called trivial if the corresponding exact sequence splits.
The extension is called essential if $\tau$ is injective, and called unital if $B$ is unital and $\tau$ is a unital $*$-homomorphism.

Two Busby invariants $\tau_i : B\to \mathcal{Q}(A\otimes\mathbb{K}), \, i=1, 2$ are said to be strongly unitarily equivalent, $\tau_1\sim_{s.u.e}\tau_2$, if there is a unitary $U\in \mathcal{M}(A\otimes\mathbb{K})$ with $\tau_1=\operatorname{Ad}\pi(U)\circ\tau_2$.
They are said to be weakly unitarily equivalent, $\tau_1\sim_{w.u.e}\tau_2$, if there is a unitary $u\in \mathcal{Q}(A\otimes\mathbb{K})$ with $\tau_1=\operatorname{Ad}u\circ\tau_2$.
We denote $\tau_1\sim_s\tau_2$ if there exist two trivial extensions $\rho_1$ and $\rho_2$ with $\tau_1\oplus\rho_1\sim_{s.u.e}\tau_2\oplus\rho_2$ and denote by $\operatorname{Ext}(B, A\otimes\mathbb{K})$ the set of the equivalence classes of the extensions with respect to the equivalence relation $\sim_s$.
\end{dfn}
We remark that $\tau_1\sim_{w.u.e}\tau_2$ implies $\tau_1\sim_s\tau_2$ (see \cite[Proposition 15.6.4]{Bl}).
For $B$ nuclear and separable, the set $\operatorname{Ext}(B, A\otimes\mathbb{K})$ is a group.

\begin{thm}[{\cite[Theorem 23.1.1]{Bl}}]\label{ext}
Let $A$ and $B$ be nuclear separable C*-algebras with $A$ in the bootstrap class.
Then there is a short exact sequence
$$0\to \bigoplus_{i=0, 1}\operatorname{Ext}_{\mathbb{Z}}^1(K_i(B), K_i(A))\to\operatorname{Ext}(B, A\otimes\mathbb{K})\to\bigoplus_{i=0, 1}\operatorname{Hom}(K_i(B), K_{i+1}(A))\to 0,$$
which splits unnaturally.
If $\bigoplus_{i=0, 1}\operatorname{Hom}(K_i(B), K_{i+1}(A))=0$, we have an isomorphism
$$\operatorname{Ext}(B, A\otimes\mathbb{K})\to \bigoplus_{i=0, 1}\operatorname{Ext}_{\mathbb{Z}}^1(K_i(B), K_i(A))$$
that sends the class of an extension $0\to A\otimes\mathbb{K}\to C\to B\to 0$ to the corresponding group extension $[K_i(A)\to K_i(C)\to K_i(B)]\in\operatorname{Ext}_{\mathbb{Z}}^1(K_i(B), K_i(A)), \, i=0, 1$.
\end{thm}

Let $E_{n+1}$ be the universal C*-algebra generated by $n+1$ isometries with mutually orthogonal ranges, called the Cuntz--Toeplitz algebra, and let $\{T_i\}_{i=1}^{n+1}$ be the canonical generators of $E_{n+1}$.
The closed two-sided ideal generated by the minimal projection $e : =1-\sum_{i=1}^{n+1}T_iT_i^*$ is isomorphic to $\mathbb{K}$, which is known to be the only closed non-trivial two-sided ideal.
Let $\pi : E_{n+1}\to \mathcal{O}_{n+1}$ be the quotient map by the ideal $\mathbb{K}$ and denote $S_i : =\pi(T_i)$.
The quotient algebra $\mathcal{O}_{n+1}$ is the universal C*-algebra generated by $n+1$ isometries with relation : $S_i^*S_j=\delta_{ij},\quad 1=\sum_{i=1}^{n+1}S_iS_i^*$.
We denote by $\mathcal{O}_{\infty}$ the universal C*-algebra generated by countably infinite isometries with mutually orthogonal ranges.
These algebras $\mathcal{O}_{n+1}, \mathcal{O}_{\infty}$ are called the Cuntz algebras with the following $K$-groups 
$$K_0(\mathcal{O}_{n+1})=\mathbb{Z}_n,\quad K_0(\mathcal{O}_{\infty})=\mathbb{Z},\quad K_1(\mathcal{O}_{n+1})=K_1(\mathcal{O}_{\infty})=0$$
(see \cite[Theorem 3.7, 3.8, Corollary 3.11]{C2}).
They are purely infinite and simple (i.e., for every non-zero element $x$, there exist two elements $y, z$ with $yxz=1$).

For a separable unital C*-algebra $A$, the short exact sequence
$0\to \mathbb{K}\otimes A\to E_{n+1}\otimes A\to \mathcal{O}_{n+1}\otimes A\to 0$
yields the following 6-term exact sequence 
$$\xymatrix{
K_0(A)\ar[r]^{-n}&K_0(A)\ar[r]&K_0(\mathcal{O}_{n+1}\otimes A)\ar[d]^{exp}\\
K_1(\mathcal{O}_{n+1}\otimes A)\ar[u]^{\delta}&K_1(A)\ar[l]&K_1(A)\ar[l]^{-n},
}$$
where $\delta$ is the index map, and the image of $\delta$ is $\operatorname{Tor}(K_0(A), \mathbb{Z}_n)$. 

Elliott--Kucerovsky's absorption theorem is the main technical ingredient in Section 3.2. 
\begin{dfn}[{\cite{EK, JG}}]\label{PL}
An extension $0\to A\otimes\mathbb{K}\to C\to B\to 0$ is called purely large if for every $c\in C \backslash A\otimes\mathbb{K}$, there is a subalgebra $D \subset \overline{c(A\otimes\mathbb{K})c^*}$ satisfying the following $\colon$

1) The subalgebra $D$ is not contained in any proper closed two-sided ideal of $A\otimes\mathbb{K}$,

2) The subalgebra $D$ is $\sigma$-unital and satisfies $D\cong D\otimes \mathbb{K}$.
\end{dfn}
Every purely large extension is essential.
The following is a special case of \cite[p 387]{EK}.
\begin{thm}\label{EKA}
Let $A$ and $B$ be separable C*-algebras, and assume that $B$ is unital and nuclear.
Let $\rho : B\to\mathcal{M}(A\otimes\mathbb{K})$ be a unital $*$-homomorphism, and let $\tau : B\to \mathcal{Q}(A\otimes\mathbb{K})$ be the Busby invariant of a unital extension $0\to A\otimes\mathbb{K}\to C\to B\to 0$.
If the extension is purely large, then we have $\tau\oplus(\pi\circ\rho)\sim_{s.u.e}\tau$.
\end{thm}
\begin{lem}
For a unital separable C*-algebra $A$, every unital extension $0\to A\otimes \mathbb{K}\to C\to \mathcal{O}_{n+1}\to 0$ is purely large.
\end{lem}

\begin{proof}
We denote $I :=1_C=1_{\mathcal{M}(A\otimes\mathbb{K})}$.
%
%
%
For $c\in C\backslash A\otimes\mathbb{K}$,
we show that there is an element $y\in\mathcal{M}(A\otimes\mathbb{K})$ with $y^*c^*cy=I$.
Since $\mathcal{O}_{n+1}$ is purely infinite and simple, there exists $x\in C$ satisfying $I-x^*c^*cx\in A\otimes\mathbb{K}$.
Let $f, h\in (A\otimes\mathbb{K})_+$ be the positive part and the negative part of the selfadjoint element $I-x^*c^*cx=f-h\in A\otimes\mathbb{K}$.
Since $(I+h)^{1/2}$ is invertible and $fh=hf=0$, one has $(I+h)^{-1/2}x^*c^*cx(I+h)^{-1/2}=I-f$.
One has $0\leq f\leq I$ because $(I+h)^{-1/2}x^*c^*cx(I+h)^{-1/2}$ is positive.
There exists a number $N$ and a positive contraction $f_0\in ((A\otimes\mathbb{M}_N)_+)_1\subset A\otimes\mathbb{K}$ with $||f_0-f||\leq 1/2$.
For the rank $N$ projection $1_N\in \mathbb{M}_N\subset\mathbb{K}$,
we have $(I-f_0)(I-(1_A\otimes 1_N))=I-(1_A\otimes 1_N)=1_A\otimes(1_{\mathcal{M}(\mathbb{K})}-1_N)$.
One can take an isometry $V\in \mathcal{M}(\mathbb{K})$ with $(1_A\otimes V)(1_A\otimes V)^*=1_A\otimes(1_{\mathcal{M}(\mathbb{K})}-1_N)$.
Direct computation yields
\begin{align*}
(1_A\otimes V)^*(I-f_0)(1_A\otimes V)&=(1_A\otimes V)^*(I-f_0)(I-(1_A\otimes 1_N))(1_A\otimes V)\\
&=(1_A\otimes V)^*(I-(1_A\otimes 1_N))(1_A\otimes V)\\
&=I.
\end{align*}
Therefore we have 
\begin{align*}
||(1_A\otimes V)^*(1-f)(1_A\otimes V)-I||&=||(1_A\otimes V)^*(1-f)(1_A\otimes V)-(1_A\otimes V)^*(1-f_0)(1_A\otimes V)||\\
&\leq ||f_0-f||\\
&\leq1/2.
\end{align*}
The selfadjoint element $(1_A\otimes V)^*(1-f)(1_A\otimes V)$ is invertible,
and $y^*c^*cy=I$ holds for $y :=x(1+h)^{-1/2}(1_A\otimes V)\{(1_A\otimes V)^*(1-f)(1_A\otimes V)\}^{-1/2}$.
Now the argument in \cite[p 425]{Lin} proves the statement.
\end{proof}

The proof of \cite[Lemma 2.14]{ST} shows the following corollary.
\begin{cor}\label{abscor}
Let $A$ be a unital separable C*-algebra, and let $\tau_i : \mathcal{O}_{n+1}\to \mathcal{Q}(A\otimes\mathbb{K}),\, i=1, 2$ be two unital extensions.
If $[\tau_1]=[\tau_2]\in \operatorname{Ext}(\mathcal{O}_{n+1}, A\otimes\mathbb{K})$, we have $\tau_1\sim_{w.u.e}\tau_2$.
\end{cor}
We refer to \cite[Proposition 2.17]{ST} for the proof of the following proposition.
\begin{prop}\label{k}
For a unital C*-algebra $A$ with $\operatorname{Tor}(K_0(A), \mathbb{Z}_n)=K_1(A)=0$ and a unital extension $\sigma : \mathcal{O}_{n+1}\to\mathcal{Q}(\mathbb{K}\otimes A)$,
we have $K_1(\sigma(\mathcal{O}_{n+1})'\cap\mathcal{Q}(\mathbb{K}\otimes A))=0$.
\end{prop}

We introduce strongly self-absorbing C*-algebras (see \cite{TW}).
\begin{dfn}
A unital separable C*-algebra $D\not =\mathbb{C}$ is called strongly self-absorbing if there exists an isomorphism $\phi : D\to D\otimes D$ and a sequence of unitaries $u_n\in D\otimes D$ with $\lim_{n\to\infty}||\phi(d)-u_n(d\otimes 1_D)u_n^*||=0$ for every $d\in D$.
\end{dfn} 
For the set $P_n$ of all prime numbers $p$ with $\operatorname{GCD}(n, p)=1$,
let $M_{(n)}$ be the UHF algebra
$$M_{(n)} : =\bigotimes_{p\in P_n}\mathbb{M}_{p^{\infty}}.$$
C*-algebras $M_{(n)}$ and $M_{(n)}\otimes \mathcal{O}_{\infty}$ are strongly self-absorbing.
\begin{thm}\label{ssa}
Let $D$ be a strongly self-absorbing C*-algebra, then the following holds $\colon$

1) There is a sequence of unital $*$-homomorphisms $\phi_n : D\to D$ with $\lim_{n\to \infty}||\phi_n(d)x-x\phi_n(d)||=0$ for every $x, d\in D$ (see \cite{TW}).

2) The algebra $D$ is $\mathcal{Z}$-stable and it is $K_1$-injective (see \cite{W, R2}).
Furthermore, the map $U(A\otimes D)/ \sim_h\to K_1(A\otimes D)$ is bijective for a separable unital C*-algebra $A$ (see \cite{XJ}). 

3) Consider $\alpha, \beta \in\operatorname{Hom}(D, A\otimes D)_u$ for a unital separable C*-algebra $A$.
There is a continuous path of unitaries $\{u_t\}_{t\in [0, 1)}\subset A\otimes D$ satisfying $u_0=1$ and $\lim_{t\to 1}||u_t\alpha(d)u_t^*-\beta(d)||=0$ for every $d\in D$ (see \cite{DW}).
\end{thm}
Let $\prod_{\mathbb{N}}A$ be the C*-algebra of all bounded sequences of $A$ and let $\bigoplus_{\mathbb{N}}A$ be the subalgebra of the sequences converging to $0$.
For a unital C*-algebra $A$, we introduce the central sequence algebra
$$A_\infty : =\left(\prod_{\mathbb{N}}A/ \bigoplus_{\mathbb{N}}A\right)\cap A',$$
where $A$ is embedded in $(\prod_{\mathbb{N}}A/ \bigoplus_{\mathbb{N}}A)$ by the constant sequence map $A\ni a\mapsto (a)_n\in \prod_{\mathbb{N}}A$.
\begin{thm}[{\cite[Theorem 2.2]{TW}}]\label{abs}
For $A$ separable unital and $D$ strongly self-absorbing,
an isomorphism $A\cong A\otimes D$ holds if and only if there is a unital embedding $D\to A_\infty.$
\end{thm}

\subsection{Some notions in the homotopy theory}
For pointed topological spaces $(X, x_0)$ and $(Y, y_0)$, we denote the set of the continuous maps from $X$ to $Y$ by $\operatorname{Map}(X, Y)$ and denote by $\operatorname{Map}_0(X, Y)$ the subset of $\operatorname{Map}(X, Y)$ consisting of the base point preserving maps.
We denote the homotopy set $\operatorname{Map}(X, Y)/\sim_h$ by $[X, Y]$ and denote $\operatorname{Map}_0(X, Y)/\sim_h$ by $[X, Y]_0$.
The $i$-th homotopy group is denoted by $\pi_i(X, x_0) : =[S^i, X]_0$.
A pointed space $(Y, y_0)$ is an $H$-space if there is a continuous map $m : Y\times Y\to Y$ such that the two maps $y\mapsto m(y, y_0)$ and $y\mapsto m(y_0, y)$ are homotopic to ${\rm id}_Y$ by the base point preserving homotopy.
We remark that the natural map $[X, Y]_0\to [X, Y]$ is bijective for a path connected $H$-space $(Y, y_0)$ and a pointed space $(X, x_0)$ with a non-degenerate base point $x_0$, in particular, for a pointed CW complex (see \cite[6.16, p 159]{AT}).
\begin{dfn}[{\cite[p 117]{AT}}]
Let $X, Y$ and $Z$ be topological spaces, and let $\pi:X\to Y$ be a continuous map.
The map $\pi$ has the homotopy lifting property (abbreviated to HLP) for $Z$, if every commuting diagram
$$
\scriptsize
\xymatrix{
\{0 \} \times Z \ar[r]^{g}\ar[d] &X\ar[d]^{\pi} \\
[0, 1]\times Z\ar[r]^{f}& Y,
}
$$
admits a continuous map $\tilde{g}\colon [0,1]\times Z\to X$ satisfying $\tilde{g}(0,z)=g(z)$ and $\pi \circ\tilde{g}=f$.
The map $\pi$ is a Serre fibration, if it has HLP for every $n$-disc.
\end{dfn}
For example, every locally trivial fiber bundle is a Serre fibration.
We refer to \cite[Lemma 2.8, 2.16, Corollary 2.9]{DP} for the proof of the following lemma.
\begin{lem}\label{se1}
Let $D$ be a unital strongly self-absorbing C*-algebra with $K_1(D)=0$.
Consider the set
$$U(E_{n+1}\otimes D)((1-e)\otimes 1_D) : =\{u((1-e)\otimes 1_D) \in E_{n+1}\otimes D\; | \;u\in U(E_{n+1}\otimes D)\}.$$
The map 
$$p \colon U(E_{n+1}\otimes D)\ni u\mapsto u((1-e)\otimes 1_D)\in U(E_{n+1}\otimes D)((1-e)\otimes 1_D)$$
gives a locally trivial fiber bundle with a fiber homeomorphic to $U(D)$.
\end{lem}
Since $K_1(E_{n+1}\otimes D)=0$, Theorem \ref{ssa} shows that the group $U(E_{n+1}\otimes D)$ and the set $U(E_{n+1}\otimes D)((1_{E_{n+1}}-e)\otimes 1_D)$  are path connected.

\begin{lem}\label{se}
A surjection $\pi\in\operatorname{Hom}(A_1, A_2)_u$ gives a Serre fibration  $\pi : U(A_1)\to U(A_2)$.
\end{lem}
\begin{proof}
For two continuous maps $f : [0, 1]\times \mathbb{D}^n\to U(A_2)$ and $g : \mathbb{D}^n\to U(A_1)$ with $f(0, x)=\pi(g(x)), x\in \mathbb{D}^n$,
we identify $f$ (resp. $g$) with a unitary of $U(C([0, 1]\times\mathbb{D}^n)\otimes A_2)$ (resp. $U(C(\mathbb{D}^n)\otimes A_1)$).
Since $f\pi(1_{C([0, 1])}\otimes g^*)\in U_0(C([0, 1]\times\mathbb{D}^n)\otimes A_2)$, there is a lift $u\in U_0(C([0, 1]\times\mathbb{D}^n)\otimes A_1)$ of $f\pi(1_{C([0, 1])}\otimes g^*)$ with $u(0, x)=1_{A_1}, x\in \mathbb{D}^n$.
The unitary $u(1_{C([0, 1])}\otimes g)\in U(C([0, 1]\times\mathbb{D}^n)\otimes A_1)$ gives a lift of $f$,
and the map $\pi$ has HLP for every $n$-disc. 
\end{proof}
\begin{thm}[{\cite[Corollary 6.44]{AT}}]\label{ex}
For a Serre fibration $\pi :(X, x_0)\to (Y, y_0)$ with the fibre $F:=\pi^{-1}(y_0)$,
there is a long exact sequence of groups $(i\geq 1)$, and exact sequence of pointed sets $(i\geq 0)$
$$\dotsm\to \pi_i(F, x_0)\to\pi_i(X, x_0)\to \pi_i(Y, y_0)\to \dotsm \to \pi_0(F, x_0)\to\pi_0(X, x_0)\to \pi_0(Y, y_0).$$
\end{thm}
For a topological group $G$,
the set of isomorphism classes of principal $G$-bundles over $X$ is identified with the homotopy set $[X, \operatorname{B}G]$, where $\operatorname{B}G$ is the classifying space, and we denote by $\operatorname{E}G\to\operatorname{B}G$ the universal principal $G$-bundle.
For example, for $H=l^2$, the space $H_1=\{f\in H\mid ||f||_{2}=1\}$ is contractible and is a model of $ES^1$ with free $S^1$ action by the scalar multiplication (see \cite{RW}). 
Identifying $\operatorname{B}S^1$ with the set of all minimal projections, the map $\operatorname{E}S^1=H_1\ni \xi \mapsto \xi\otimes\xi^*\in\operatorname{B}S^1$ gives the universal bundle where we denote by $\xi \otimes \xi^*$ the operator $H\ni x\mapsto \langle x, \xi\rangle\xi\in H$. 
Fix a base point $e\in \operatorname{B}S^1$.
The map $\operatorname{Aut}(\mathbb{K})\ni \alpha \mapsto \alpha(e) \in \operatorname{B}S^1$ is a homotopy equivalence, and the following Dixmier--Douady theorem holds.
\begin{thm}[{\cite{DD, CS, RW}}]
For a finite CW complex $X$,
the set $[X, \operatorname{BAut}(\mathbb{K})]$, identified with the set of isomorphism classes  of the locally trivial continuous fields of $\mathbb{K}$,
has a multiplication given by the fiber wise tensor product of the fields, and it is isomorphic to $H^3(X)$.
\end{thm}

\subsection{The Dadarlat--Pennig theory}
Using the isomorphism $\phi : D\cong D\otimes D$,
one gives $K_0(C(X)\otimes D)$ a ring structure whose multiplication comes from the following map
$$\Delta_X : (C(X)\otimes D)^{\otimes 2}\ni f_1(x)\otimes f_2(y)\mapsto \phi^{-1}(f_1(x)\otimes f_2(x))\in C(X)\otimes D$$
(see \cite[Section 2.4]{DP}).
We denote the set of the invertible elements by $K_0(C(X)\otimes D)^\times$.
\begin{thm}[{\cite[Theorem 2.22, 3.8, Lemma 2.8]{DP}}]\label{DPT}
Let $D\not =\mathbb{C}$ be a strongly self-absorbing C*-algebra and let $X$ be a compact metrizable space.
We denote by $\operatorname{Aut}_0(\mathbb{K}\otimes D)$ the path component of $\operatorname{Aut}(\mathbb{K}\otimes D)$ containing ${\rm id}_{\mathbb{K}\otimes D}$ and denote by $\mathcal{P}_0(\mathbb{K}\otimes D)$ the set of projections which are homotopy equivalent to $e\otimes 1_D$.

1) For two continuous maps $\alpha, \beta : X\to \operatorname{Aut}(\mathbb{K}\otimes D)$,
which are identified with the $C(X)$-linear isomorphisms of $C(X)\otimes\mathbb{K}\otimes D$,
one has
$$K_0(\Delta_X)\circ K_0(\alpha\otimes\beta)([(1_{C(X)}\otimes e\otimes 1_D)^{\otimes 2}]_0)=K_0(\alpha\circ\beta)([1_{C(X)}\otimes e\otimes 1_D]_0),$$
and the following map is multiplicative $\colon$
$$[X, \operatorname{Aut}(\mathbb{K}\otimes D)]\ni[\alpha]\mapsto [\alpha(e\otimes 1_D)]_0\in K_0(C(X)\otimes \mathbb{K}\otimes D).$$
The image of the map is $K_0(C(X)\otimes D)_+\cap K_0(C(X)\otimes D)^\times$.

2) The following map $\eta$ is a homotopy equivalence and induces an isomorphism $\eta_*$ for connected $X$$\colon$
$$\eta : \operatorname{Aut}_0(\mathbb{K}\otimes D)\ni \alpha\mapsto\alpha(e\otimes 1_D)\in \mathcal{P}_0(\mathbb{K}\otimes D),$$
$$\eta_* : [X, \operatorname{Aut}_0(\mathbb{K}\otimes D)]\to 1+K_0(C_0(X, x)\otimes D).$$

3) The homotopy set $E^1_D(X) : =[X, \operatorname{BAut}(\mathbb{K}\otimes D)]$ has a group structure defined by the fiber wise tensor product of the locally trivial continuous fields of $\mathbb{K}\otimes D$.
\end{thm}
\subsection{The homotopy groups of $\operatorname{Aut}(\mathcal{O}_{n+1})$}
Using the well-known homeomorphism
$$u : \operatorname{End}(\mathcal{O}_{n+1})\ni \alpha\mapsto \sum_{i=1}^{n+1}\alpha(S_i)S_i^*\in U(\mathcal{O}_{n+1})$$
M. Dadarlat computed the homotopy groups of $\operatorname{Aut}(\mathcal{O}_{n+1})$.
\begin{thm}[{\cite[Theorem 1.1]{D2}}]\label{L}
The inclusion map $\operatorname{Aut}(\mathcal{O}_{n+1})\to \operatorname{End}(\mathcal{O}_{n+1})$ is a weak homotopy equivalence,
and we have $$\pi_{\rm odd}(\operatorname{Aut}(\mathcal{O}_{n+1}))=\pi_{\rm odd}(\operatorname{End}(\mathcal{O}_{n+1}))=\mathbb{Z}_n,\quad \pi_{\rm even}(\operatorname{Aut}(\mathcal{O}_{n+1}))=\pi_{\rm even}(\operatorname{End}(\mathcal{O}_{n+1}))=0.$$
\end{thm}

\section{The weak homotopy equivalence results}
Our goal in this section is to prove Corollary \ref{mt3}.
Throughout this section, we assume that $D$ is a strongly self-absorbing C*-algebra satisfying $\operatorname{Tor}(K_0(D))=K_1(D)=0$ and $\mathcal{O}_{n+1}\otimes D\cong \mathcal{O}_{m+1}$ for some $m\geq 1$.
Note that the above conditions automatically hold if $D$ satisfies the UCT (see \cite[Proposition 5.1]{TW}).
We denote by $\operatorname{End}_0(E_{n+1}\otimes D)$ the path component of $\operatorname{End}(E_{n+1}\otimes D)$  containing ${\rm id}_{E_{n+1}\otimes D}$.
Since $\beta(e\otimes 1_D)\sim_h e\otimes 1_D$ for $\beta\in \operatorname{End}_0(E_{n+1}\otimes D)$,
one has $\beta(\mathbb{K}\otimes D)\subset \mathbb{K}\otimes D$, and we have a semi-group homomorphism $$q : \operatorname{End}_0(E_{n+1}\otimes D)\ni \beta\mapsto \tilde{\beta}\in \operatorname{End}(\mathcal{O}_{n+1}\otimes D).$$
\subsection{The weak homotopy type of $\operatorname{End}_0(E_{n+1}\otimes D)$}
\begin{thm}\label{mt}
The following map is a weak homotopy equivalence $\colon$
$$\operatorname{End}_0(E_{n+1}\otimes D)\ni \beta\mapsto \sum_{i=1}^{n+1}\beta(T_i\otimes 1_D)(T_i^*\otimes 1_D)\in U(E_{n+1}\otimes D)((1_{E_{n+1}}-e)\otimes 1_D).$$
\end{thm}
Note that the above map is well-defined (see \cite[Lemma 2.11]{ST}).
By the isomorphism $l :=\phi^{-1} : D^{\otimes 2}\to D$,
it is enough to show that the following map is a weak homotopy equivalence :
$$\Pi : \operatorname{End}_0(E_{n+1}\otimes D^{\otimes 2})\ni \beta\mapsto \sum_{i=1}^{n+1}\beta(T_i\otimes 1_{D^{\otimes 2}})(T_i^*\otimes 1_{D^{\otimes 2}})\in U(E_{n+1}\otimes D^{\otimes 2})((1_{E_{n+1}}-e)\otimes 1_{D^{\otimes 2}}).$$
By Theorem \ref{ssa}, there is a path of unitaries $\{v_t\}_{t\in [0, 1)}\subset U(D^{\otimes 2})$ satisfying $v_0=1_{D^{\otimes 2}}$ and $\lim_{t\to 1}||\operatorname{Ad}v_t(d)-l(d)\otimes 1_D||=0$ for every $d\in D^{\otimes 2}$.
Let $H_t$ be the map
$$H_t : \operatorname{End}_0(E_{n+1}\otimes D^{\otimes 2})\ni \beta \mapsto ({\rm id}_{E_{n+1}}\otimes(\operatorname{Ad}v_t^*\circ(l\otimes 1_D))) \circ \beta \in \operatorname{End}_0(E_{n+1}\otimes D^{\otimes 2})$$
for $t\in [0, 1)$ and $H_1(\rho) :=\rho$.
Let $h_t$ be the map
$$h_t : U(E_{n+1}\otimes D^{\otimes 2})((1_{E_{n+1}}-e)\otimes 1_{D^{\otimes 2}})\to U(E_{n+1}\otimes D^{\otimes 2})((1_{E_{n+1}}-e)\otimes 1_{D^{\otimes 2}})$$
defined by $h_t(w) : =({\rm id}_{E_{n+1}}\otimes(\operatorname{Ad}v_t^*\circ(l\otimes 1_D)))(w)$ for $t\in [0, 1)$ and $h_1(w) :=w$.
One can check that $H_t(\beta)$ converges to $\beta$ as $t$ tends to $1$ in the point-wise norm topology,  $\lim_{t\to 1}||h_t(w)-w||=0$ and $\Pi\circ H_t=h_t\circ\Pi$.
For $w\in U(E_{n+1}\otimes D^{\otimes 2})((1_{E_{n+1}}-e)\otimes 1_{D^{\otimes 2}})$,
we define a map $\varrho_w\in \operatorname{Hom}(E_{n+1}, E_{n+1}\otimes D^{\otimes 2})_u$ by $\varrho_w(T_i) : =w(T_i\otimes 1_{D^{\otimes 2}})$.
Since $h_0(w)\in E_{n+1}\otimes D\otimes \mathbb{C}1_D$,
the image of $\varrho_{h_0(w)}$ is in $E_{n+1}\otimes D\otimes \mathbb{C}1_D$, and we have a $*$-endomorphism
$$\alpha_{w} : E_{n+1}\otimes D^{\otimes 2}\ni f\otimes d\mapsto \varrho_{h_0(w)}(f)(1_{E_{n+1}\otimes D}\otimes l(d))\in (E_{n+1}\otimes D)\otimes D.$$
Since $U(E_{n+1}\otimes D^{\otimes 2})((1_{E_{n+1}}-e)\otimes 1_{D^{\otimes 2}})$ is path connected,
Theorem \ref{ssa} implies ${\rm id}_{E_{n+1}\otimes D^{\otimes 2}}\sim_h{\rm id}_{E_{n+1}}\otimes 1_D\otimes l\sim_h \alpha_{w}$,
and one has $\alpha_{w}\in\operatorname{End}_0(E_{n+1}\otimes D^{\otimes 2})$ and $\Pi(\alpha_{w})=h_0(w)$.
\begin{lem}\label{he}
We have $h_t(w)\in \operatorname{Im}\Pi$ for $w\in U(E_{n+1}\otimes D^{\otimes 2})((1_{E_{n+1}}-e)\otimes 1_{D^{\otimes 2}})$, $t\in [0, 1)$.
The inclusion map $\operatorname{Im}\Pi\hookrightarrow U(E_{n+1}\otimes D^{\otimes 2})((1_{E_{n+1}}-e)\otimes 1_{D^{\otimes 2}})$
is a homotopy equivalence.
\end{lem}
\begin{proof}
Direct computation yields
\begin{align*}
h_t(w)&=\operatorname{Ad}(1_{E_{n+1}}\otimes v_t^*)\circ h_0(w)\\
&=(1_{E_{n+1}}\otimes v_t^*)\Pi(\alpha_{w})(1_{E_{n+1}}\otimes v_t)\\
&=\sum_{i=1}^{n+1}(1_{E_{n+1}}\otimes v_t^*)\alpha_{w}(T_i\otimes 1_{D^{\otimes 2}})(T_i^*\otimes 1_{D^{\otimes 2}})(1_{E_{n+1}}\otimes v_t)\\
&=\Pi(\operatorname{Ad}(1_{E_{n+1}}\otimes v_t^*)\circ\alpha_{w}).
\end{align*}
This shows $h_t(w)\in \operatorname{Im}\Pi$ for $t\in [0, 1)$,
and the last part of the lemma follows immediately.
\end{proof}
Let us choose ${\rm id}_{E_{n+1}}\otimes 1_D\otimes l$ (resp. $(1_{E_{n+1}}-e)\otimes 1_{D^{\otimes 2}}$) as a base point of $\operatorname{End}_0(E_{n+1}\otimes D^{\otimes 2})$ (resp. $U(E_{n+1}\otimes D^{\otimes 2})((1_{E_{n+1}}-e)\otimes 1_{D^{\otimes 2}})$).
Since $h_t$ fixes the base point,
one has $[\gamma]=[h_0\circ\gamma]\in \pi_k(U(E_{n+1}\otimes D^{\otimes 2})((1_{E_{n+1}}-e)\otimes 1_{D^{\otimes 2}}))$, and the following lemma holds.
\begin{lem}\label{q}
The following  map is surjective $\colon$
$$\Pi_* : \pi_k(\operatorname{End}_0(E_{n+1}\otimes D^{\otimes 2}))\to \pi_k(U(E_{n+1}\otimes D^{\otimes 2})((1_{E_{n+1}}-e)\otimes 1_{D^{\otimes 2}})).$$
\end{lem}
For $\Gamma\in \operatorname{Map}_0(S^k, \operatorname{End}_0(E_{n+1}\otimes D^{\otimes 2}))$ with $\Pi_*([\Gamma])=0$,
we define another map $\Gamma'$ by $\Gamma'_x :=\alpha_{\Pi(\Gamma_x)}$.
\begin{lem}
We have $\Gamma\sim_h\Gamma'$ in $\operatorname{Map}(S^k, \operatorname{End}_0(E_{n+1}\otimes D^{\otimes 2}))$.
\end{lem}
\begin{proof}
Note that $\Gamma$ is homotopic to $H_0\circ\Gamma$.
Regarding $\varrho_{h_0(\Pi(\Gamma_x))}$, $x\in S^k$ as a $C(S^k)$-linear map in $\operatorname{Hom}(E_{n+1}, C(S^k)\otimes E_{n+1}\otimes D^{\otimes 2})_u$ and identifying it by $\varrho_{h_0(\Pi(\Gamma))}$,
we define a unital separable C*-algebra $P_\Gamma : =(\operatorname{Im}\varrho_{h_0(\Pi(\Gamma))})'\cap C(S^k)\otimes E_{n+1}\otimes D^{\otimes 2}$.
Since $1_{C(S^k)\otimes E_{n+1}\otimes D}\otimes D\subset P_\Gamma$,
Theorem \ref{ssa}, \ref{abs} show an isomorphism $P_\Gamma\cong P_\Gamma\otimes D$.
By Theorem \ref{ssa},
there is a path $\{\theta_t\}_{t\in [0, 1]}\subset \operatorname{Hom}(D^{\otimes 2}, P_\Gamma)_u$ with $\theta_1(d)=1_{C(S^k)\otimes E_{n+1}\otimes D}\otimes l(d)$,
$$\theta_0(d) : S^k\ni x\mapsto H_0\circ \Gamma_x(1_{E_{n+1}}\otimes d),\quad d\in D^{\otimes 2}.$$
The following map gives a homotopy to prove $\Gamma'\sim_h H_0\circ\Gamma$ :
$$\Gamma_t : E_{n+1}\otimes D^{\otimes 2}\ni f\otimes d\mapsto \varrho_{h_0(\Pi(\Gamma))}(f)\theta_t(d)\in C(S^k)\otimes E_{n+1}\otimes D^{\otimes 2},\quad t\in [0, 1].$$
\end{proof}
\begin{proof}[{Proof of Theorem \ref{mt}}]
By Lemma \ref{q},
it is enough to show the injectivity of the map $\Pi_*$.
Since $\operatorname{End}_0(E_{n+1}\otimes D^{\otimes 2})$ is a path connected $H$-space,
the map $\pi_k(\operatorname{End}_0(E_{n+1}\otimes D^{\otimes 2}))\to [S^k, \operatorname{End}_0(E_{n+1}\otimes D^{\otimes 2})]$ is bijective as mentioned before,
and we show that $\Gamma'$ is homotopic to a constant map of ${\rm id}_{E_{n+1}}\otimes 1_D\otimes l$.
By the assumption of $\Gamma$,
there is a path $\{w_t\}_{t\in [0, 1]}\subset \operatorname{Map}(S^k, (U(E_{n+1}\otimes D^{\otimes 2})((1_{E_{n+1}}-e)\otimes 1_{D^{\otimes 2}}))$ with $w_0=\Pi\circ \Gamma$, $w_1\equiv (1_{E_{n+1}}-e)\otimes 1_{D^{\otimes 2}}$.
A homotopy from $\Gamma'$ to the constant map is given by $\alpha_{w_t}, t\in [0, 1]$.
\end{proof}
Recall the homeomorphism $u : \operatorname{End}(\mathcal{O}_{n+1})\ni\sigma\mapsto \sum_{i=1}^{n+1}\sigma(S_i)S_i^*\in U(\mathcal{O}_{n+1})$.
\begin{lem}\label{we}
The following map is a weak homotopy equivalence $\colon$
$$u' : \operatorname{End}(\mathcal{O}_{n+1}\otimes D)\ni\sigma\mapsto \sum_{i=1}^{n+1}\sigma(S_i\otimes 1_D)(S_i^*\otimes 1_D)\in U(\mathcal{O}_{n+1}\otimes D).$$
\end{lem}
\begin{proof}
The unital $*$-homomorphism $\otimes 1_D : \mathcal{O}_{n+1}\ni a\mapsto a\otimes 1_D\in \mathcal{O}_{n+1}\otimes D\cong \mathcal{O}_{m+1}$ induces a surjection $K_{k-1}(\mathcal{O}_{n+1})\to K_{k-1}(\mathcal{O}_{n+1}\otimes D)$,
and the map $u'_*$ is surjective by the following diagram
$$\xymatrix{
\pi_k(\operatorname{End}(\mathcal{O}_{n+1}\otimes D))\ar[r]^{u'_*}&\pi_k(U(\mathcal{O}_{n+1}\otimes D))\ar@{=}[r]&K_1(C(S^k)\otimes\mathcal{O}_{n+1}\otimes D)\\
\pi_k(\operatorname{End}(\mathcal{O}_{n+1}))\ar[u]^{{\otimes {\rm id}_D}_*}\ar@{=}[r]^{u_*}&\pi_k(U(\mathcal{O}_{n+1}))\ar[u]^{{\otimes 1_D}_*}\ar@{=}[r]&K_1(C(S^k)\otimes\mathcal{O}_{n+1}).\ar@{->>}[u]^{K_1({\otimes 1_D})}
}$$
The equation $|\pi_k(\operatorname{End}(\mathcal{O}_{n+1}\otimes D))|=|\pi_k(U(\mathcal{O}_{n+1}\otimes D))|(=m, 0)$ shows the injectivity.
\end{proof}
\begin{thm}\label{mt1}
The map $q : \operatorname{End}_0(E_{n+1}\otimes D)\to\operatorname{End}(\mathcal{O}_{n+1}\otimes D)$ is a weak homotopy equivalence.
\end{thm}
\begin{proof}
The following diagram commutes
$$
\scriptsize
\xymatrix{
p^{-1}((1_{E_{n+1}}-e)\otimes 1_D)\ar[r]\ar[d]^{i}&U(E_{n+1}\otimes D)\ar[r]^{p\quad\quad\quad\quad}\ar@{=}[d]&U(E_{n+1}\otimes D)((1_{E_{n+1}}-e)\otimes 1_D)\ar[d]^{\pi}&\ar[l]^{\quad\quad\quad\quad Thm \ref{mt}}\ar[d]^{q}\operatorname{End}_0(E_{n+1}\otimes D)\\
\pi^{-1}(1_{\mathcal{O}_{n+1}\otimes D})\ar[r]&U(E_{n+1}\otimes D)\ar[r]^{\pi}&U(\mathcal{O}_{n+1}\otimes D)&\ar[l]^{u'_*}\operatorname{End}(\mathcal{O}_{n+1}\otimes D)
}$$
where two horizontal sequences contain the Serre fibrations in Lemma \ref{se}, \ref{se1}.
By 2) of Theorem \ref{ssa}, the inclusion $i$ is a weak homotopy equivalence,
and the vertical map $\pi$ is a weak homotopy equivalence by 5-lemma.
Theorem \ref{mt} and Lemma \ref{we} prove the statement.
\end{proof}


\subsection{The weak homotopy type of $\operatorname{Aut}(E_{n+1}\otimes D)$}
\begin{lem}\label{pcc}
The image of the restriction map $\eta_n : \operatorname{Aut}(E_{n+1}\otimes D)\to \operatorname{Aut}(\mathbb{K}\otimes D)$ is contained in $\operatorname{Aut}_0(\mathbb{K}\otimes D)$.
\end{lem}
\begin{proof}
Note that $\operatorname{Tor}(K_0(D))=0$ by our assumption.
Since $K_0(\alpha)={\rm id}_{K_0(E_{n+1}\otimes D)}$ for $\alpha\in\operatorname{Aut}(E_{n+1}\otimes D)$
and the map $K_0(\mathbb{K}\otimes D)\xrightarrow{-n} K_0(E_{n+1}\otimes D)$ is injective,
one has $[\alpha(e\otimes 1_D)]_0=[e\otimes 1_D]_0\in K_0(\mathbb{K}\otimes D)$,
and Theorem \ref{DPT} proves the statement.
\end{proof}
\begin{lem}\label{triv}
For a finite CW complex $X$ with $\operatorname{Tor}(K^0(X), \mathbb{Z}_n)=0$, we have $\operatorname{Im}{\eta_n}_*=\{0\}\subset [X, \operatorname{Aut}_0(\mathbb{K}\otimes D)]$ for the map ${\eta_n}_* : [X, \operatorname{Aut}(E_{n+1}\otimes D)]\to [X, \operatorname{Aut}_0(\mathbb{K}\otimes D)]$.
\end{lem}
\begin{proof}
We identify a continuous map $\alpha : X\to \operatorname{Aut}(E_{n+1}\otimes D)$ with a $C(X)$-linear $*$-isomorphism of $C(X)\otimes E_{n+1}\otimes D$.
For $$u'\circ\tilde{\alpha}=\sum_{i=1}^{n+1}\tilde{\alpha}(1_{C(X)}\otimes S_i\otimes 1_D)(1_{C(X)}\otimes S_i^*\otimes 1_D)\in U(C(X)\otimes \mathcal{O}_{n+1}\otimes D),$$
we have the following lift of the unitary $u'\circ\tilde{\alpha}\oplus (u'\circ\tilde{\alpha})^*$ :
$$\scriptsize\left(
\begin{array}{cc}
\sum_{i=1}^{n+1}\alpha(1\otimes T_i\otimes 1_D)(1\otimes T_i^*\otimes 1_D)& \alpha(1\otimes e\otimes 1_D)\\
1\otimes e\otimes 1_D& \sum_{i=1}^{n+1}(1\otimes T_i\otimes 1_D)\alpha(1\otimes T_i^*\otimes 1_D)
\end{array}
\right)\in\mathbb{M}_2(C(X)\otimes E_{n+1}\otimes D),
$$
and direct computation yields
$$[1_{C(X)}\otimes e\otimes 1_D]_0-[\alpha(1_{C(X)}\otimes e\otimes  1_D)]_0=\delta([u'\circ\tilde{\alpha}]_1)\in \operatorname{Tor}(K^0(X)\otimes K_0(D), \mathbb{Z}_{n})=0.$$
Therefore Theorem \ref{DPT} proves the statement.
\end{proof}

For an arbitrary $\beta\in \operatorname{Aut}(E_{n+1}\otimes D)$,
we consider the following $C(S^1)$-algebras
$$C_\beta :=\{ F\in C[0, 1]\otimes\mathbb{K}\otimes D\; | \; F(0)=\beta(F(1))\},$$
$$M_\beta :=\{ F\in C[0, 1]\otimes E_{n+1}\otimes D \;|\; F(0)=\beta(F(1))\},$$
$$A_{\tilde{\beta}}:=\{ a\in C[0, 1]\otimes \mathcal{O}_{n+1}\otimes D \;|\; a(0)=\tilde{\beta}(a(1))\},$$
$$C(S^1)\otimes \mathcal{O}_{n+1}\otimes D=\{ a\in C[0, 1]\otimes \mathcal{O}_{n+1}\otimes D\;|\; a(0)=a(1)\},$$
$$C(S^1)\otimes \mathbb{K}\otimes D =\{ F\in C[0,1]\otimes \mathbb{K}\otimes D\; |\; F(0)=F(1)\}.$$
By Lemma \ref{pcc}, there is a path $\Xi : [0, 1]\to \operatorname{Aut}_0(\mathbb{K}\otimes D)$ from $\Xi_0={\rm id}_{\mathbb{K}\otimes D}$ to $\Xi_1=\eta_n(\beta)$.
So one has an isomorphism
$$\theta_\beta : C_\beta\ni F(t)\mapsto \Xi_t(F(t))\in C(S^1)\otimes \mathbb{K}\otimes D.$$
Since $C_\beta\subset M_\beta$ is an essential ideal, the map $\theta_\beta$ induces an essential unital extension $\tau_\beta : A_{\tilde{\beta}}\to \mathcal{Q}(C(S^1)\otimes\mathbb{K}\otimes D)$.
There is a path $\xi : [0, 1]\to \operatorname{Aut}(\mathcal{O}_{n+1}\otimes D)$ with $\xi_0={\rm id}_{\mathcal{O}_{n+1}\otimes D}, \xi_1=\tilde{\beta}$,
and one has a $C(S^1)$-linear isomorphism $\phi_{\tilde{\beta}} : A_{\tilde{\beta}}\ni a(t)\mapsto \xi_t(a(t))\in C(S^1)\otimes\mathcal{O}_{n+1}\otimes D$.
We define a unital embedding by $$\sigma_\beta  : \mathcal{O}_{n+1}\otimes D\ni d \mapsto \tau_\beta(\phi^{-1}_{\tilde{\beta}}(1_{C(S^1)}\otimes d))\in \mathcal{Q}(C(S^1)\otimes\mathbb{K}\otimes D).$$
By definition, two Busby invariants ${\rm ev}_0\circ \sigma_\beta$ and ${\rm ev}_0\circ\sigma_{{\rm id}_{E_{n+1}\otimes D}}$ are equal where we denote by ${\rm ev}_0 : \mathcal{Q}(C(S^1)\otimes\mathbb{K}\otimes D)\to\mathcal{Q}(\mathbb{K}\otimes D)$ the induced map by the evaluation at $0$.

\begin{lem}\label{sue}
We have $\sigma_\beta\sim_{s.u.e}\sigma_{{\rm id}_{E_{n+1}\otimes D}}$.
\end{lem}
\begin{proof}
Using Theorem \ref{ext}, the isomorphism ${{\rm ev}_0}_* : \operatorname{Ext}(\mathcal{O}_{n+1}\otimes D, C(S^1)\otimes \mathbb{K}\otimes D)\cong \operatorname{Ext}(\mathcal{O}_{n+1}\otimes D, \mathbb{K}\otimes D)$ yields $[\sigma_\beta]=[\sigma_{{\rm id}_{E_{n+1}\otimes D}}]\in \operatorname{Ext}(\mathcal{O}_{n+1}\otimes D, C(S^1)\otimes \mathbb{K}\otimes D)$.
By Corollary \ref{abscor}, there is a unitary $w\in U(\mathcal{Q}(C(S^1)\otimes\mathbb{K}\otimes D))$ with $\sigma_\beta =\operatorname{Ad}w\circ\sigma_{{\rm id}_{E_{n+1}\otimes D}}$.
One has $${\rm ev_0}\circ\sigma_{{\rm id}_{E_{n+1}\otimes D}}=\operatorname{Ad}{\rm ev_0}(w)\circ({\rm ev}_0\circ\sigma_{{\rm id}_{E_{n+1}\otimes D}}),$$
and Proposition \ref{k} shows $[{\rm ev}_0(w)]_1=0\in K_0(\mathcal{Q}(\mathbb{K}\otimes D))$ that implies $[w]_1=0\in K_1(\mathcal{Q}(C(S^1)\otimes \mathbb{K}\otimes D))$.
Now $K_1$-injectivity of $\mathcal{Q}(C(S^1)\otimes \mathbb{K}\otimes D)$ proves the statement.
\end{proof}
\begin{thm}\label{pc}
The group $\operatorname{Aut}(E_{n+1}\otimes D)$ is path connected.
\end{thm}
\begin{proof}
By Lemma \ref{sue}, there is a $C(S^1)$-linear isomorphism $M_\beta\cong C(S^1)\otimes E_{n+1}\otimes D$,
and this gives a homotopy from $\beta$ to ${{\rm id}_{E_{n+1}\otimes D}}$.
\end{proof}

To prove our main technical result,
it remains to show that the inclusion $\operatorname{Aut}(E_{n+1}\otimes D)\hookrightarrow\operatorname{End}_0(E_{n+1}\otimes D)$ is a weak homotopy equivalence.
\begin{lem}\label{surject}
The map $\pi_{2m-1}(\operatorname{Aut}(E_{n+1}\otimes D))\to\pi_{2m-1}(\operatorname{End}_0(E_{n+1}\otimes D))$ is surjective.
\end{lem}
\begin{proof}
By \cite[Lemma 3.4, Theorem 3.14]{ST} and \cite[Theorem 7.4]{D2},
the map $\pi_{2m-1}(\operatorname{End}_0(E_{n+1}))\to\pi_{2m-1}(\operatorname{End}(\mathcal{O}_{n+1}))$ is surjective.
By the diagram in the proof of Lemma \ref{we}, the map ${\otimes {\rm id}_D}_* : \pi_{2m-1}(\operatorname{End}(\mathcal{O}_{n+1}))\to\pi_{2m-1}(\operatorname{End}(\mathcal{O}_{n+1}\otimes D))$ is surjective.
Now Theorem \ref{mt1} and \cite[Theorem 3.14]{ST} yield the following commutative diagram
$$\xymatrix{
\pi_{2m-1}(\operatorname{Aut}(E_{n+1}\otimes D))\ar[r]&\pi_{2m-1}(\operatorname{End}_0(E_{n+1}\otimes D))\ar@{=}[r]^{Thm \ref{mt1}}&\pi_{2m-1}(\operatorname{End}(\mathcal{O}_{n+1}\otimes D))\\
\pi_{2m-1}(\operatorname{Aut}(E_{n+1}))\ar[u]^{{\otimes {\rm id}_D}_*}\ar@{=}[r]&\pi_{2m-1}(\operatorname{End}_0(E_{n+1}))\ar[u]^{{\otimes{\rm id}_D}_*}\ar@{>>}[r]&\pi_{2m-1}(\operatorname{End}(\mathcal{O}_{n+1})),\ar@{>>}[u]^{Lem \ref{we}}
}$$
and diagram chasing proves the statement.
\end{proof}
Since Theorem \ref{mt1} yields $\pi_{2m}(\operatorname{End}_0(E_{n+1}\otimes D))=0$, the above lemma shows the surjectivity of the map $\pi_k(\operatorname{Aut}(E_{n+1}\otimes D))\to\pi_k(\operatorname{End}_0(E_{n+1}\otimes D)), k\geq 0$.
We refer to \cite[Lemma 2.8, 2.16, Corollary 2.9]{DP} for the proof of the following lemma.
\begin{lem}\label{fib}
The map $\operatorname{End}_0(E_{n+1}\otimes D)\ni \beta\mapsto \beta(e\otimes 1_D)\in \mathcal{P}_0(\mathbb{K}\otimes D)$,
gives two fibrations
$$\operatorname{Aut}_{e\otimes 1_D}(E_{n+1}\otimes D)\to\operatorname{Aut}(E_{n+1}\otimes D) \to\mathcal{P}_0(\mathbb{K}\otimes D),$$
$$\operatorname{End}_{e\otimes 1_D}(E_{n+1}\otimes D)\to\operatorname{End}_0(E_{n+1}\otimes D)\to\mathcal{P}_0(\mathbb{K}\otimes D)$$
where $\operatorname{Aut}_{e\otimes 1_D}(E_{n+1}\otimes D)$ $($resp. $\operatorname{End}_{e\otimes 1_D}(E_{n+1}\otimes D)$$)$ is a subset of $\operatorname{Aut}(E_{n+1}\otimes D)$ $($resp. $\operatorname{End}_0(E_{n+1}\otimes D)$$)$ consisting of $*$-automorphisms fixing $e\otimes 1_D$.
\end{lem}
\begin{lem}
For $\alpha : S^{k}\to \operatorname{Aut}(E_{n+1}\otimes D)$ with $[\alpha]=0\in \pi_{k}(\operatorname{End}_0(E_{n+1}\otimes D))$,
there exists a map $\alpha'$ satisfying $[\alpha]=[\alpha']\in \pi_{k}(\operatorname{Aut}(E_{n+1}\otimes D))$ and $[\alpha']=0\in\pi_{k}(\operatorname{End}_{e\otimes 1_D}(E_{n+1}\otimes D))$.
\end{lem}
\begin{proof}
Lemma \ref{fib} and Theorem \ref{mt1} yield two long exact sequences of the homotopy groups and the following commutative diagram

$$
\scriptsize
\xymatrix{
\pi_{2m}(\operatorname{Aut}(E_{n+1}\otimes D))\ar[d]\ar[r]&K_0(D)\ar@{=}[d]\ar[r]&\pi_{2m-1}(\operatorname{Aut}_{e\otimes 1_D}(E_{n+1}\otimes D))\ar[d]\ar[r]&\pi_{2m-1}(\operatorname{Aut}(E_{n+1}\otimes D))\to 0\ar[d]\\
0\ar[r]&K_0(D)\ar[r]&\pi_{2m-1}(\operatorname{End}_{e\otimes 1_D}(E_{n+1}\otimes D))\ar[r]&\pi_{2m-1}(\operatorname{End}_0(E_{n+1}\otimes D))\to 0.
}$$
One has a map $\alpha'' : S^{2m-1}\to\operatorname{Aut}_{e\otimes 1_D}(E_{n+1}\otimes D)$ with $[\alpha]=[\alpha'']\in\pi_{2m-1}(\operatorname{Aut}(E_{n+1}\otimes D))$.
There is an element $a\in K_0(D)$ which is sent to $[\alpha'']\in \pi_{2m-1}(\operatorname{End}_{e\otimes 1_D}(E_{n+1}\otimes D))$,
and one can find a map $\beta : S^{2m-1}\to \operatorname{Aut}_{e\otimes 1_D}(E_{n+1}\otimes D)$ to which $-a$ is sent.
Since $[\beta]=0\in \pi_{2m-1}(\operatorname{Aut}(E_{n+1}\otimes D))$,
one has $[\alpha]=[\alpha'']=[\alpha']\in\pi_{2m-1}(\operatorname{Aut}(E_{n+1}\otimes D))$ for the map $\alpha' : S^{2m-1}\ni x\mapsto \beta_x\circ\alpha''_x\in\operatorname{Aut}_{e\otimes 1_D}(E_{n+1}\otimes D)$.
It is also checked that $[\alpha']=-a+a=0\in \pi_{2m-1}(\operatorname{End}_{e\otimes 1_D}(E_{n+1}\otimes D))$.
The map $\pi_{2m}(\operatorname{Aut}_{e\otimes 1_D}(E_{n+1}\otimes D))\to \pi_{2m}(\operatorname{Aut}(E_{n+1}\otimes D))$ is surjective by diagram chasing,
and we can prove the statement because $\pi_{2m}(\operatorname{End}_{e\otimes 1_D}(E_{n+1}\otimes D))=0$.
\end{proof}
The proof of \cite[Lemma 3.3]{ST} shows $\pi_{\rm even}(\operatorname{Aut}(E_{n+1}\otimes D))=0$.
Following the argument in \cite[Lemma 3.7]{ST}, we can determine the weak unitary equivalence class of the Busby invariant associated with $\alpha'$ and the same computation as in \cite[Lemma 3.8, 3.9, 3.10, Theorem 3. 11]{ST} proves $[\alpha']=0\in \pi_{\rm odd}(\operatorname{Aut}(E_{n+1}\otimes D))$.
Therefore the following theorem holds.
\begin{thm}\label{ry}
The inclusion $\operatorname{Aut}(E_{n+1}\otimes D)\hookrightarrow\operatorname{End}_0(E_{n+1}\otimes D)$ is a weak homotopy equivalence.
\end{thm}
Combining Theorem \ref{L}, \ref{mt1}, \ref{ry}, we have the following.
\begin{cor}\label{mt3}
The group homomorphism $q : \operatorname{Aut}(E_{n+1}\otimes D)\to \operatorname{Aut}(\mathcal{O}_{n+1}\otimes D)$ is a weak homotopy equivalence.
In particular,
the map $\operatorname{B}(q) : \operatorname{BAut}(E_{n+1}\otimes D)\to \operatorname{BAut}(\mathcal{O}_{n+1}\otimes D)$ is a weak homotopy equivalence and gives the bijection $\operatorname{B}(q)_*$ of the homotopy sets.
\end{cor}

\section{A topological invariant}
By \cite[Theorem 1.1]{D2}, all the continuous fields over finite CW-complexes are locally trivial,
and we identify them with the principal $\operatorname{Aut}(\mathcal{O}_{n+1})$ bundles.
\subsection{The invariant $\mathfrak{b}_D$}
\begin{dfn}\label{I}
For a finite CW complex $X$, we define $\mathfrak{b}_D : [X, \operatorname{BAut}(\mathcal{O}_{n+1})]\to E_D^1(X)$ by
\scriptsize$$[X, \operatorname{BAut}(\mathcal{O}_{n+1})]\xrightarrow{\operatorname{B}(\otimes {\rm id}_D)_*}[X, \operatorname{BAut}(\mathcal{O}_{n+1}\otimes D)]\xrightarrow{\operatorname{B}(q)_*^{-1}}[X, \operatorname{BAut}(E_{n+1}\otimes D)]\xrightarrow{\operatorname{B}(\eta_n)_*}[X, \operatorname{BAut}_0(\mathbb{K}\otimes D)]\subset E^1_D(X).$$
\end{dfn}
\begin{thm}\label{imageb}
The image of $\mathfrak{b}_D$ consists of $n^k$-torsion elements of $E_D^1(X)$ for some $k\geq 1$.
\end{thm}
\begin{proof}
Since $\operatorname{Aut}(\mathcal{O}_2)$ is contractible (see \cite{DP}) and $\mathcal{O}_{n+1}\otimes \mathbb{M}_{n^\infty}\cong \mathcal{O}_2$, the composition $\operatorname{B}(\otimes{\rm id}_{\mathbb{M}_{n^\infty}})_*\circ\mathfrak{b}_D=\mathfrak{b}_{D\otimes \mathbb{M}_{n^\infty}}$ is the 0 map.
Therefore the statement follows from \cite[Theorem 2.11]{DP2}.
\end{proof}
If ${\rm dim}X\leq 3$, or $X=SY$,
the map $\mathfrak{b}_D$ is given by the Bockstein map,
and hence the invariant is non-trivial (see Section 5).
\subsection{The invariant $\mathfrak{b}_{(n)}$}
We denote $\mathfrak{b}_{(n)}:=\mathfrak{b}_{M_{(n)}}$ for short.
For $r\geq 1$, with $\operatorname{GCD}(n, r)=1$,
the Kirchberg--Phillips theorem yields an isomorphism $\varphi_r : \mathcal{O}_{n+1}\to \mathcal{O}_{nr+1}\otimes M_{(n)}$,
and one has the following map
\scriptsize$$b_r : [X, \operatorname{BAut}(\mathcal{O}_{n+1})]\xrightarrow{\operatorname{B}(\operatorname{Ad}\varphi_r)_*}[X, \operatorname{BAut}(\mathcal{O}_{nr+1}\otimes M_{(n)})]\xrightarrow{\operatorname{B}(q)_*^{-1}}[X, \operatorname{BAut}(E_{nr+1}\otimes M_{(n)})]\xrightarrow{\operatorname{B}(\eta_{nr})_*}[X, \operatorname{BAut}(\mathbb{K}\otimes M_{(n)})].$$
\small
\begin{thm}\label{r1}
The map $b_r$ is equal to $b_1$.
\end{thm}
We identify $E_{nr+1}\otimes M_{(n)}$ with a subalgebra of $\mathcal{M}(\mathbb{K}\otimes M_{(n)})$ by a unital embedding $l_r : E_{nr+1}\otimes M_{(n)}\hookrightarrow\mathcal{M}(\mathbb{K}\otimes M_{(n)})$ which sends $e\otimes 1_{M_{(n)}}$ into $\mathcal{P}_0(\mathbb{K}\otimes M_{(n)})$.
We need the following proposition to prove Theorem \ref{r1}.
\begin{prop}\label{isom}
There is an automorphism $\beta_{\frac{1}{r}}\in\operatorname{Aut}(\mathbb{K}\otimes M_{(n)})$ inducing an automorphism of $\mathcal{M}(\mathbb{K}\otimes M_{(n)})$ that restricts to an isomorphism $\beta_{\frac{1}{r}} : E_{n+1}\otimes M_{(n)} \to E_{nr+1}\otimes M_{(n)}$.
\end{prop}
To show the proposition, we need the following two lemmas.
The embedding $l_r$ gives the Busby invariant $i_r : \mathcal{O}_{nr+1}\otimes M_{(n)}\to \mathcal{Q}(\mathbb{K}\otimes M_{(n)})$.
By the Kirchberg--Phillips theorem, the isomorphism $\varphi_r$ is determined uniquely upto homotopy,
and the Busby invariant $\tau_r :=i_r\circ\varphi_r$ is defined.
\begin{lem}
In $\operatorname{Ext}(\mathcal{O}_{n+1}, \mathbb{K}\otimes M_{(n)})\cong\operatorname{Ext}^1_{\mathbb{Z}}(\mathbb{Z}_n, \mathbb{Z}_{(n)})$, we have $[\tau_r]=\frac{1}{r}[\tau_1]$.
\end{lem}
\begin{proof}
By Theorem \ref{ext}, the element $[\tau_r]$ corresponds to the extension of groups $[\mathbb{Z}_{(n)}\xrightarrow{-nr}\mathbb{Z}_{(n)}\to \mathbb{Z}_n]$.
If we identify $[\mathbb{Z}_{(n)}\xrightarrow{n}\mathbb{Z}_{(n)}\to \mathbb{Z}_n]$ with $\bar{1}\otimes 1\in \mathbb{Z}_n\otimes \mathbb{Z}_{(n)}$, then one has $[\tau_r]=\bar{1}\otimes \frac{-1}{r}$.
\end{proof}
By Theorem \ref{DPT},
there is an automorphism $\alpha_{\frac{1}{r}}\in\operatorname{Aut}(\mathbb{K}\otimes M_{(n)})$ with $[\alpha_{\frac{1}{r}}(e\otimes 1_{M_{(n)}})]_0=\frac{1}{r}\in K_0(M_{(n)})$.
For the induced automorphism $\tilde{\alpha_{\frac{1}{r}}} \in\operatorname{Aut}(\mathcal{Q}(\mathbb{K}\otimes M_{(n)}))$,
one has $$[\tilde{\alpha_{\frac{1}{r}}}\circ\tau_1]={\alpha_{\frac{1}{r}}}_*([\tau_1])=\frac{1}{r}[\tau_1]=[\tau_r].$$
By Corollary \ref{abscor}, we have a unitary $v\in\mathcal{Q}(\mathbb{K}\otimes M_{(n)})$ with $\tilde{\alpha_{\frac{1}{r}}}\circ\tau_1=\operatorname{Ad}v\circ\tau_r$.
We would like to show that $v$ lifts to an element in $U(\mathcal{M}(\mathbb{K}\otimes M_{(n)}))$.
For $\tau_r$,
we denote Paschke's unitary by
$$V_{\tau_r} : =\left(
\begin{array}{cc}
0_1&\tau_r(\mathbb{S})\\
\pi(w)&{\rm O}_{n+1}
\end{array}\right)\in\mathbb{M}_{n+2}(\mathcal{Q}(\mathbb{K}\otimes M_{(n)})),$$
where $w\in \mathbb{M}_{n+2}(\mathcal{M}(\mathbb{K}\otimes M_{(n)}))$ is a partial isometry with $ww^*=0_1\oplus 1_{n+1}, w^*w=1_1\oplus 0_{n+1}$ and $\mathbb{S} :=(S_1, \cdots, S_{n+1})$ (see \cite{P}).
\begin{lem}\label{exx}
We have $r\delta([V_{\tau_r}]_1)=\delta([V_{\tau_1}]_1)=-1\in K_0(M_{(n)})$.
\end{lem}
\begin{proof}
For $\tau_r^{\oplus r} :=1_r \otimes\tau_r : \mathcal{O}_{n+1}\to\mathbb{M}_r \otimes\mathcal{Q}(\mathbb{K}\otimes M_{(n)})$,
one has $1_r \otimes\tau_r=({\rm id}_{\mathbb{M}_r}\otimes i_r)\circ(1_r\otimes \varphi_r)$ by definition,
and we denote $\sigma_0 :=1_r\otimes \varphi_r : \mathcal{O}_{n+1}\to \mathbb{M}_r\otimes\mathcal{O}_{nr+1}\otimes M_{(n)}$.

Let $T^{(k)}_1, \cdots , T^{(k)}_r, (k=1, \cdots, n), T_{nr+1}$ be the canonical $nr+1$ generators of $E_{nr+1}$.
We define the following $n+1$ isometries with mutually orthogonal ranges :
$$\mathcal{T}_k :=\left(
\begin{array}{ccc}
T^{(k)}_1&\cdots&T^{(k)}_r\\
&\text{\rm {\huge{O}}}&\\
\end{array}\right)\in \mathbb{M}_r(E_{nr+1}),\quad (k=1,\cdots, n),\quad\mathcal{T}_{n+1} :=T_{nr+1}\oplus 1_{r-1}.$$
Direct computation yields $\sum_{i=1}^{n+1}\mathcal{T}_i\mathcal{T}_i^*=(1-e)\oplus 1_{r-1}$,
and we have a unital $*$-homomorphism $$\sigma_1 : \mathcal{O}_{n+1}\ni S_i\mapsto \pi(\mathcal{T}_i)\otimes 1_{M_{(n)}}\in\mathbb{M}_r(\mathcal{O}_{nr+1}\otimes M_{(n)}).$$
By the Kirchberg--Phillips Theorem, one has $\sigma_1\sim_h\sigma_0$ in $\operatorname{Hom}(\mathcal{O}_{n+1}, \mathbb{M}_r(\mathcal{O}_{nr+1}\otimes M_{(n)}))_u$.
Therefore we have $\tau_r^{\oplus r}=({\rm id}_{\mathbb{M}_r}\otimes i_r)\circ\sigma_0\sim_h({\rm id}_{\mathbb{M}_r}\otimes i_r)\circ\sigma_1$ and $[V_{\tau_r^{\oplus r}}]_1=[V_{({\rm id}_{\mathbb{M}_r}\otimes i_r)\circ\sigma_1}]_1$.
The following unitary is a lift of $V_{({\rm id}_{\mathbb{M}_r}\otimes i_r)\circ\sigma_1}\oplus V^*_{({\rm id}_{\mathbb{M}_r}\otimes i_r)\circ\sigma_1}$
$$\left(
\begin{array}{cccc}
0_1&\mathcal{T}&(e\otimes 1_{M_{(n)}})\oplus 0_{r-1}& 0\\
1_r\otimes w&{\rm O}_{n+1}&0&{\rm O}_{n+1}\\
& &0_1&1_r\otimes w^*\\
&\text{\rm {\huge{O}}}_{n+2}&\mathcal{T}^*&{\rm O}_{n+1}
\end{array}
\right)\in \mathbb{M}_{2n+4}(\mathbb{M}_r\otimes E_{nr+1}\otimes M_{(n)}),$$
where we denote $\mathcal{T} :=(\mathcal{T}_1\otimes 1_{M_{(n)}},\cdots, \mathcal{T}_{n+1}\otimes 1_{M_{(n)}})$ for simplicity.
By the definition of the index map, we have $r\delta([V_{\tau_r}]_1)=\delta([V_{({\rm id}_{\mathbb{M}_r}\otimes i_r)\circ\sigma_1}]_1)=-[e\otimes 1_{M_{(n)}}]_0=-1\in K_0(M_{(n)})$.
\end{proof}
\begin{proof}[{Proof of Proposition \ref{isom}}]
Direct computation yields
\begin{align*}
-n\delta([v]_1)=&\delta([V_{\operatorname{Ad}v\circ\tau_r}]_1)-\delta([V_{\tau_r}]_1)\\
=&\delta([V_{\tilde{\alpha_{\frac{1}{r}}}\circ\tau_1}]_1)-\delta([V_{\tau_r}]_1)\\
=&\delta\circ K_1(\tilde{\alpha_{\frac{1}{r}}})([V_{\tau_1}]_1)-\delta([V_{\tau_r}]_1)\\
=&\frac{1}{r}\delta([V_{\tau_1}]_1)-\delta([V_{\tau_r}]_1)=0.
\end{align*}
So we have a unitary $V$ with $\pi(V)=v$,
and the map $\beta_{\frac{1}{r}}:=\operatorname{Ad}V^*\circ\alpha_{\frac{1}{r}}$ is the isomorphism.
\end{proof}
\begin{proof}[{Proof of Theorem \ref{r1}}]
Every arrow in the following diagram is a group homomorphism :
$$\xymatrix{
\operatorname{Aut}(\mathcal{O}_{n+1})\ar@{=}[d]\ar[r]^{\operatorname{Ad}\varphi_r\quad}&\operatorname{Aut}(\mathcal{O}_{nr+1}\otimes M_{(n)})&\operatorname{Aut}(E_{nr+1}\otimes M_{(n)})\ar[l]\ar[r]^{\eta_{nr}}&\operatorname{Aut}(\mathbb{K}\otimes M_{(n)})\\
\operatorname{Aut}(\mathcal{O}_{n+1})\ar[r]^{\operatorname{Ad}\varphi_1\quad}&\operatorname{Aut}(\mathcal{O}_{n+1}\otimes M_{(n)})\ar[u]^{\operatorname{Ad}\tilde{\beta_{\frac{1}{r}}}}&\operatorname{Aut}(E_{n+1}\otimes M_{(n)})\ar[l]\ar[u]^{\operatorname{Ad}\beta_{\frac{1}{r}}}\ar[r]^{\eta_n}&\operatorname{Aut}(\mathbb{K}\otimes M_{(n)}).\ar[u]^{\operatorname{Ad}\beta_{\frac{1}{r}}}
}$$
Since $K_1(\mathcal{O}_{nr+1}\otimes M_{(n)})=0$,
the Kirchberg--Phillips theorem gives a path of unitaries $\{u_t\}_{t\in [0, 1)}\subset \mathcal{O}_{nr+1}\otimes M_{(n)}$ satisfying $u_0=1$ and that $\operatorname{Ad}u_t\circ\varphi_r$ converges to $\tilde{\beta_{\frac{1}{r}}}\circ\varphi_1$ as $t$ tends to $1$.
So the left hand square of the above diagram commutes upto homotopy and the homotopy is given by a path of group homomorphisms.
In particular, the following diagram commutes for a finite CW complex $X$, where $\operatorname{B}(\operatorname{Ad}\beta_{\frac{1}{r}})_*={\rm id}_{E^1_{M_{(n)}}(X)}$ because $\beta_{\frac{1}{r}}\in\operatorname{Aut}(\mathbb{K}\otimes M_{(n)})$ :
$$
\scriptsize\xymatrix{
[X, \operatorname{BAut}(\mathcal{O}_{n+1})]\ar@{=}[d]\ar[r]^{\operatorname{B}(\operatorname{Ad}\varphi_r)_*\quad}&[X, \operatorname{BAut}(\mathcal{O}_{nr+1}\otimes M_{(n)})]&[X, \operatorname{BAut}(E_{nr+1}\otimes M_{(n)})]\ar@{=}[l]^{Cor \ref{mt3}}\ar[r]^{\operatorname{B}(\eta_r)_*}&[X, \operatorname{BAut}(\mathbb{K}\otimes M_{(n)})]\\
[X, \operatorname{BAut}(\mathcal{O}_{n+1})]\ar[r]^{\operatorname{B}(\operatorname{Ad}\varphi_1)_*\quad}&[X, \operatorname{BAut}(\mathcal{O}_{n+1}\otimes M_{(n)})]\ar[u]^{\operatorname{B}(\operatorname{Ad}\tilde{\beta_{\frac{1}{r}}})_*}&[X, \operatorname{BAut}(E_{n+1}\otimes M_{(n)})]\ar@{=}[l]^{Cor \ref{mt3}}\ar[u]^{\operatorname{B}(\operatorname{Ad}\beta_{\frac{1}{r}})_*}\ar[r]^{\operatorname{B}(\eta_1)_*}&[X, \operatorname{BAut}(\mathbb{K}\otimes M_{(n)})].\ar[u]^{\operatorname{B}(\operatorname{Ad}\beta_{\frac{1}{r}})_*}
}$$
\end{proof}
\begin{thm}
Two maps $\operatorname{B}(\operatorname{Ad}\varphi_1)_*$ and $\operatorname{B}(\otimes {\rm id}_{M_{(n)}})_*$ are equal,
and we have $b_r=\mathfrak{b}_{{(n)}}$.
\end{thm}
\begin{proof}
For the isomorphism $\phi : M_{(n)}\to M_{(n)}^{\otimes 2}, l=\phi^{-1}$,
we have a path of unitaries $\{v_t\}_{t\in [0, 1)}\subset U(M_{(n)}^{\otimes 2})$ satisfying $v_0=1$ and $\lim_{t\to 1}||\operatorname{Ad}v_t(d)-l(d)\otimes 1_{M_{(n)}}||=0$ for every $d\in M_{(n)}^{\otimes 2}$.
Identifying $\mathcal{O}_{n+1}$ (resp. $\mathcal{O}_{n+1}\otimes M_{(n)}$) with $\mathcal{O}_{n+1}\otimes M_{(n)}$ (resp. $(\mathcal{O}_{M_{(n)}}\otimes M_{(n)})\otimes M_{(n)}$),
we identify $\varphi_1$ with ${\rm id}_{\mathcal{O}_{n+1}}\otimes \phi$ and define the following automorphism for $\alpha\in \operatorname{Aut}(\mathcal{O}_{n+1}\otimes M_{(n)})$
$$\Psi_t(\alpha) :=\operatorname{Ad}(1_{\mathcal{O}_{n+1}}\otimes v^*_t)\circ (\alpha\otimes {\rm id}_{M_{(n)}})\circ\operatorname{Ad}(1_{\mathcal{O}_{n+1}}\otimes v_t)\in \operatorname{Aut}((\mathcal{O}_{n+1}\otimes M_{(n)})\otimes M_{(n)}).$$
The map $\Psi_t$ is a group homomorphism for every $t\in [0,1)$ and $\Psi_t(\alpha)$ converges to $({\rm id}_{\mathcal{O}_{n+1}}\otimes \phi)\circ \alpha\circ({\rm id}_{\mathcal{O}_{n+1}}\otimes \phi)^{-1}$ as $t$ tends to $1$ in the point norm topology.
So $\Psi_t$ gives the homotopy of group homomorphisms,
and one has $\operatorname{B}(\operatorname{Ad}\varphi_1)_*=\operatorname{B}(\otimes {\rm id}_{M_{(n)}})_*$.
\end{proof}


\subsection{The inverse image $\mathfrak{b}_{{(n)}}^{-1}(0)$}
In this section we assume that $X$ is a finite connected CW complex.
Let us recall the continuous fields of $\mathcal{O}_{n+1}$ coming from vector bundles (see \cite{D1, r2}).
For a compact Hausdorff space $X$, we denote by ${\rm Vect}_m\;(X)$ the set of the vector bundles of rank $m$. 
M. Dadarlat investigated continuous fields of $\mathcal{O}_{n+1}$ over $X$ arising 
from $E\in {\rm Vect}_{n+1}(X)$, which are Cuntz--Pimsner algebras. 
We refer to \cite{KT} and \cite{Pim} for Cuntz--Pimsner algebras. 
Fixing a Hermitian structure of $E$, we get a Hilbert $C(X)$-module from $E$, which we regard as a $C(X)$-$C(X)$-bimodule.  
Then the Pimsner construction gives the Cuntz--Pimsner algebra $\mathcal{O}_E$ and the exact sequence $0\to \mathcal{K}_E\xrightarrow{j_E}\mathcal{T}_E\to \mathcal{O}_E\to 0.$ 
Since $\mathcal{K}_E$ has a projection whose range is the 1-dimensional subspace of the vacuum vector, the Dixmier--Douady theory yields $\mathcal{K}_E\cong C(X)\otimes \mathbb{K}$.
The algebra $\mathcal{O}_E$ (resp. $\mathcal{T}_E$) is a continuous field of $\mathcal{O}_{n+1}$ (resp. $E_{n+1}$) over $X$ with the natural unital inclusion $\theta_E : C(X)\to \mathcal{O}_E$.

\begin{thm}[{\cite[Theorem 4.8]{Pim}}]\label{op}
The exact sequence $0\to \mathcal{K}_E\xrightarrow{j_E}\mathcal{T}_E\to \mathcal{O}_E\to 0$ gives the exact sequence of $K_0$-groups $\colon$
$$K^0(X)\xrightarrow{1-[E]}K^0(X)\xrightarrow{K_0(\theta_{E})}K_0(\mathcal{O}_E).$$
\end{thm}

M. Dadarlat found an invariant to classify the $C(X)$-linear isomorphism classes of $\mathcal{O}_E$.
\begin{thm}[{\cite[Theorem 1.1]{D1}}]\label{dada}
Let $X$ be a compact metrizable space, and let $E$ and $F$ be vector bundles of rank $\geq 2$ over $X$. Then there is a unital $*$-homomorphism $\varphi : \mathcal{O}_E\to \mathcal{O}_F$ with $\varphi\circ\theta_E=\theta_F$ if and only if $(1-[E])\cdot K^0(X)\subset(1-[F])\cdot K^0(X)$. Moreover we can take $\varphi$ to be an isomorphism if and only if $(1-[E])\cdot K^0(X)=(1-[F])\cdot K^0(X)$.
\end{thm}
We show the following theorem.
\begin{thm}\label{mt4}
For a finite connected CW complex $X$,
the inverse image $\mathfrak{b}_{{(n)}}^{-1}(0)$ consists of the $C(X)$-linear isomorphism classes of the continuous fields of the form $\mathcal{O}_E\otimes M_{(n)}$ for $E\in \operatorname{Vect}_{nr+1}(X)$ with $r\geq 1, \operatorname{GCD}(n, r)=1$.
\end{thm}
All necessary arguments are already in \cite{D1}.
We use Kasparov's parametrized $KK$-groups $KK_X(\cdot, \cdot)$ for $C(X)$-algebras (see \cite{KK}).
Consider two unital separable continuous $C(X)$-algebras $A, B$ with the maps $\theta_A, \theta_B$ that determine the $C(X)$-linear structure.
By \cite[Theorem 1.1, Theorem 2.7]{D3},
a $KK_X$-equivalence $\sigma\in KK_X(A, B)$ with $[1_A]_0\otimes \sigma =[1_B]_0$ lifts to a $C(X)$-linear isomorphism $A\cong B$ where we identifies $KK_X(C(X), A)\ni KK_X(\theta_A)$ with $KK(\mathbb{C}, A)\ni [1_A]_0$.
Note that for two continuous fields $C(X)\otimes C, B$, one has the isomorphism of groups $KK_X(C(X)\otimes C, B)=KK(C, B)$ (see \cite[Proof of Corollary 2.8]{D3}).
If there is a $KK_X$-equivalence $\mu\in KK_X(SA, SB)$ with $KK_X(S\theta_A)\otimes \mu=KK_X(S\theta_B)$,
the suspension isomorphism $KK_X(A, B)\cong KK_X(SA, SB)$ gives the above $\sigma\in KK_X(A, B)$.
Since \cite{HRW} implies $M_{(n)}\otimes \mathcal{O}_\infty$-stability of the continuous fields, the proposition below implies Theorem \ref{mt4}.
\begin{prop}\label{Ker}
Let $\mathcal{O}\in \mathfrak{b}_{{(n)}}^{-1}(0)$ be a continuous field with the map $\theta : C(X)\to \mathcal{O}$ that determines the $C(X)$-algebra structure.
Then there is a number $r\geq 1, \operatorname{GCD}(n, r)=1$ and a rank $nr+1$ vector bundle $F$,
and we have a $KK_X$-equivalence $$\mu \in KK_X(S\mathcal{O}\otimes (M_{(n)}\otimes \mathcal{O}_{\infty}),\; S\mathcal{O}_F\otimes M_{(n)}\otimes (M_{(n)}\otimes \mathcal{O}_{\infty}))$$ with $KK_X(S(\theta\otimes 1))\otimes \mu=KK_X(S(\theta_F\otimes 1))$.
\end{prop}
For two $C(X)$-algebras $A, B$ and a $C(X)$-linear $*$-homomorphism $\varphi : A\to B$,
one has a mapping cone algebra
$$C_\varphi :=\{(f, a)\in (C_0(0, 1]\otimes B)\oplus A\; |\; f(1)=\varphi(a)\},$$
and the following Puppe sequence holds (see \cite{Bl}) :
$$0\to SB\hookrightarrow C_\varphi \to A\to 0.$$
If $A\subset B$ is an ideal and $\varphi$ is the inclusion, one has the exact sequence $$0\to C_0(0, 1]\otimes A\to C_{\varphi}\xrightarrow{q_{\varphi}} S(B/A)\to 0,$$
and the quotient map $q_{\varphi}$ is a $KK_X$-equivalence because $C_0(0, 1]\otimes A$ is $C(X)$-linearly contractible.
We need the following lemma.
\begin{lem}[{\cite[Appendix A]{MN}}]\label{mn}
Let $A, A', B, B'$ be $C(X)$-algebras and let $\varphi : A\to B, \;\varphi' : A'\to B'$ be $C(X)$-linear $*$-homomorphisms.
Consider an additive category $KK_X$ of $C(X)$-algebras whose morphism $\operatorname{Mor}(A, B)$ is given by $KK_X(A, B)$.
For $\alpha\in KK_X(A, A'),\;\beta\in KK_X(B, B')$, we consider the following diagram
$$\xymatrix{
SB\ar[r]\ar[d]^{S\beta}&C_\varphi\ar[r]&A\ar[r]^{\varphi}\ar[d]^{\alpha}&B\ar[d]^{\beta}\\
SB'\ar[r]&C_{\varphi'}\ar[r]&A'\ar[r]^{\varphi'}&B'.
}$$
If the diagram commutes in the category $KK_X$,
then there exists $\gamma\in KK_X(C_\varphi, C_{\varphi'})$ that makes the diagram commutes.
Furthermore, if $\alpha, \beta$ are $KK_X$-equivalences, one can chose $\gamma$ to be a $KK_X$-equivalence.
\end{lem}
\begin{proof}[{Proof of Proposition \ref{Ker}}]
For every $\mathcal{O}\in \mathfrak{b}_{{(n)}}^{-1}(0)$,
there is a locally trivial continuous $C(X)$-algebra $\mathcal{E}$ of $E_{n+1}\otimes M_{(n)}$, and the following exact sequence of $C(X)$-algebras holds
$$0\to C(X)\otimes \mathbb{K}\otimes M_{(n)}\xrightarrow{j} \mathcal{E}\to \mathcal{O}\to 0.$$
The evaluation map ${\rm ev}_x$ at $x\in X$ gives an extension $0\to\mathbb{K}\otimes M_{(n)}\to E_{n+1}\otimes M_{(n)}\to \mathcal{O}_{n+1}\otimes M_{(n)}\to 0$.
We also denote by $\theta : C(X)\to \mathcal{E}$ the map that determines the $C(X)$-linear structure.
By \cite[Theorem 1.1]{D3}, the map $\theta\otimes {\rm id} : C(X)\otimes (M_{(n)}\otimes \mathcal{O}_{\infty})\to \mathcal{E}\otimes (M_{(n)}\otimes \mathcal{O}_{\infty})$ gives a $KK_X$-equivalence.
Since $$K_0({\rm ev}_x)([(j\otimes {\rm id}_{M_{(n)}\otimes \mathcal{O}_{\infty}})(1_{C(X)}\otimes e\otimes 1_{M_{(n)}^{\otimes 2}\otimes \mathcal{O}_{\infty}})]_0)=-n\in K_0(E_{n+1}\otimes M_{(n)}^{\otimes 2}\otimes \mathcal{O}_{\infty}),$$
there exist $r\geq 1, \operatorname{GCD}(n, r)=1$ and $y\in \tilde{K}^0(X)$ satisfying $$[(j\otimes {\rm id}_{M_{(n)}\otimes \mathcal{O}_{\infty}})(1_{C(X)}\otimes e\otimes 1_{M_{(n)}^{\otimes 2}\otimes \mathcal{O}_{\infty}})]_0=-(n+\frac{y}{r})\in K_0(\mathcal{E}\otimes M_{(n)}\otimes \mathcal{O}_{\infty})=K^0(X)\otimes \mathbb{Z}_{(n)}.$$
The Kirchiberg--Phillips theorem gives a $*$-homomorphism $M_{(n)}\otimes \mathcal{O}_{\infty}\ni 1\mapsto p\in C(X)\otimes M_{(n)}\otimes \mathcal{O}_\infty$ with $[p]_0=-(n+\frac{y}{r})$ providing a $C(X)$-linear $*$-homomorphism $\alpha_{-(n+\frac{y}{r})} : C(X)\otimes M_{(n)}\otimes \mathcal{O}_{\infty} \to C(X)\otimes M_{(n)}\otimes \mathcal{O}_{\infty}$ with $[\alpha_{-(n+\frac{y}{r})}(1)]_0=-(n+\frac{y}{r})$.
Let $e\otimes$ be the map
$$e\otimes : C(X)\otimes M_{(n)}\otimes \mathcal{O}_\infty \ni f \mapsto (e \otimes 1_{M_{(n)}})\otimes f\in (C(X)\otimes \mathbb{K}\otimes M_{(n)})\otimes M_{(n)}\otimes \mathcal{O}_\infty$$
that gives a $KK_X$-equivalence.
By the identification 
\begin{align*}
KK_X(C(X)\otimes (M_{(n)}\otimes \mathcal{O}_\infty), \; \mathcal{E}\otimes M_{(n)}\otimes \mathcal{O}_{\infty})&\cong KK(M_{(n)}\otimes \mathcal{O}_\infty, \mathcal{E}\otimes M_{(n)}\otimes \mathcal{O}_{\infty})\\
&=\operatorname{Hom}(\mathbb{Z}_{(n)}, K^0(X)\otimes \mathbb{Z}_{(n)}),
\end{align*}
one has $KK_X((j\otimes {\rm id})\circ (e\otimes))=-(n+\frac{y}{r})=KK_X((\theta\otimes {\rm id})\circ (\alpha_{-(n+\frac{y}{r})}))$, 
and the following diagram commutes in $KK_X$
$$\xymatrix{
(C(X)\otimes \mathbb{K}\otimes M_{(n)})\otimes M_{(n)}\otimes \mathcal{O}_\infty\ar[r]^{\quad\quad j\otimes {\rm id}}&\mathcal{E}\otimes M_{(n)}\otimes \mathcal{O}_\infty\\
C(X)\otimes M_{(n)}\otimes \mathcal{O}_\infty\ar[u]^{e\otimes }\ar[r]^{\quad\quad\alpha_{-(n+\frac{y}{r})}}&C(X)\otimes M_{(n)}\otimes \mathcal{O}_\infty.\ar[u]^{\theta\otimes {\rm id}}
}$$
There is  a vector bundle $F\in\operatorname{Vect}_{nrR+1}(X)$ satisfying $[F]-(nrR+1)=yR\in \tilde{K}^0(X)$ for sufficiently large $R\geq 1, \operatorname{GCD}(n, R)=1$.
Let $\alpha_{\frac{1}{rR}}$ (resp. $\alpha_{-(nrR+yR)}$) be a $*$-endomorphism of $C(X)\otimes M_{(n)}\otimes \mathcal{O}_{\infty}$ with $[\alpha_{\frac{1}{rR}}(1)]_0=\frac{1}{rR}$ (resp. $[\alpha_{-(nrR+yR)}(1)]_0=-(nrR+yR)=1-[F]$).
Now we have the following diagram commuting in $KK_X$
$$\xymatrix{
&S\mathcal{O}\otimes D&&\\
S\mathcal{E}\otimes D\ar[r]&C_{j\otimes {\rm id}}\ar[r]\ar[u]^{q_{j\otimes {\rm id}}}&C(X)\otimes\mathbb{K}\otimes M_{(n)}\otimes D\ar[r]^{\quad\quad\quad j\otimes {\rm id}}&\mathcal{E}\otimes D\\
SC(X)\otimes D\ar[u]^{S(\theta\otimes {\rm id})}\ar[r]&C_{\alpha_{-(n+\frac{y}{r})}}\ar@{.>}[u]\ar@{.>}[d]\ar[r]&C(X)\otimes D\ar[d]^{\alpha_{\frac{1}{rR}}}\ar[u]^{e\otimes}\ar[r]^{\alpha_{-(n+\frac{y}{r})}}&C(X)\otimes D\ar[u]^{\theta\otimes {\rm id}}\\
SC(X)\otimes D\ar[d]^{S(\theta_F\otimes 1_{M_{(n)}})\otimes {\rm id}}\ar@{=}[u]\ar[r]&C_{\alpha_{-(nrR+yR)}}\ar@{.>}[d]\ar[r]&C(X)\otimes D\ar[d]^{e\otimes}\ar[r]^{\alpha_{-(nrR+yR)}}&C(X)\otimes D\ar@{=}[u]\ar[d]^{(\theta_F\otimes 1_{M_{(n)}})\otimes {\rm id}}\\
S\mathcal{T}_F\otimes M_{(n)}\otimes D\ar[r]&C_{j_F\otimes {\rm id}}\ar[r]\ar[d]^{q_{j_F\otimes {\rm id}}}&C(X)\otimes \mathbb{K}\otimes M_{(n)}\otimes D\ar[r]^{\quad 1-[F]}&\mathcal{T}_F\otimes M_{(n)}\otimes D\\
&S\mathcal{O}_F\otimes M_{(n)}\otimes D&&
}$$
where we denote $D=M_{(n)}\otimes \mathcal{O}_\infty$ for simplicity.
By Lemma \ref{mn}, the same argument as in \cite[Proof of Theorem 1.1]{D1} proves the statement.
\end{proof}
\begin{rem}\label{numb}
For two vector bundles $F\in \operatorname{Vect}_{nR+1}(X)$ and $F'\in\operatorname{Vect}_{nR'+1}(X)$ with $R, R'\geq 1, \operatorname{GCD}(n, R)=\operatorname{GCD}(n, R')=1$,
the same argument in the above proof proves that $\mathcal{O}_F\otimes M_{(n)}\cong \mathcal{O}_{F'}\otimes M_{(n)}$ if and only if $(1-[F])K^0(X)\otimes \mathbb{Z}_{(n)}=(1-[F'])K^0(X)\otimes \mathbb{Z}_{(n)}$.
Therefore we can identify $\mathfrak{b}_{{(n)}}^{-1}(0)$ with the set of equivalence classes $\tilde{K}^0(X)\otimes \mathbb{Z}_{(n)}/\sim_n$ where the equivalence relation is defined as follows $\colon$ $a\sim_n b \Leftrightarrow (n+a)(1+z)=n+b$ for some $z\in \tilde{K}^0(X)\otimes \mathbb{Z}_{(n)}$ (see \cite[Lemma 4.7]{r2}).
\end{rem}
\begin{thm}\label{po}
For a finite CW complex $X$ with $\operatorname{Tor}(H^{2k+1}(X), \mathbb{Z}_n)=0$ for every $k\geq 1$,
we have $\operatorname{Im}\mathfrak{b}_{(n)}=0$ and $[X, \operatorname{BAut}(\mathcal{O}_{n+1})]=\tilde{K}^0(X)\otimes \mathbb{Z}_{(n)}/\sim_n$.
\end{thm}
\begin{proof}
By Remark \ref{numb} and Theorem \ref{imageb},
it is enough to show $\operatorname{Tor}(\bar{E}_{M_{(n)}}^1(X), \mathbb{Z}_n)=0$.
The assumption yields that the group $H^{2k+1}(X, \mathbb{Z}_{(n)})$ is torsion free for every $k\geq 1$,
and the graded subgroups of $\bar{E}_{M_{(n)}}^1(X)$ associated with the Atiyah--Hirzebruch spectral sequence are all torsion free by the argument of \cite[Corollary 4.4]{DP}.
In particular, we have $\operatorname{Tor}(\bar{E}^1_{M_{(n)}}(X), \mathbb{Z}_n)=0$.
\end{proof}
This generalizes \cite[Theorem 1.6]{D1} and \cite[Theorem 4.12]{r2},
and one can classify the continuous fields over the space whose even cohomology admits $n$-torsion.
For example, Theorem \ref{po} applies to the real projective spaces $\mathbb{R}{\rm P}^m$ for $m\geq 2$.
\begin{rem}
Following the argument in \cite[Proof of Theorem 2.11]{DP2},
we can identify $\mathfrak{b}_{\mathcal{Z}}\otimes 1 : [X, \operatorname{BAut}(\mathcal{O}_{n+1})]\to \bar{E}^1_{\mathcal{Z}}(X)\otimes \mathbb{Z}_{(n)}$ with $\mathfrak{b}_{(n)}=\mathfrak{b}_{\mathcal{Z}\otimes M_{(n)}}$ and we have $\mathfrak{b}_{\mathcal{Z}}^{-1}(0)=\mathfrak{b}_{(n)}^{-1}(0)$.
It also follows that $\mathfrak{b}_{\mathcal{O}_\infty}^{-1}(0)=\mathfrak{b}_{(n)}^{-1}(0)$.
\end{rem}

\section{Examples and questions}
\subsection{Examples}
We give some examples of computation of the invariant $\mathfrak{b}_{{(n)}}$.
\subsubsection{The case of ${\rm dim}X\leq 3$}
By Corollary \ref{mt3} and \cite[Proof of Lemma 3.20]{ST},
the group homomorphism $\otimes {\rm id}_{M_{(n)}} : \operatorname{Aut}(E_{n+1})\to\operatorname{Aut}(E_{n+1}\otimes M_{(n)})$ is 3-connected,
and the map $\operatorname{B}(\otimes {\rm id}_{M_{(n)}})$ is 4-connected.
Construction of the Postnikov tower yields another $4$-connected map $P_2 : \operatorname{BAut}(E_{n+1})\to K(\mathbb{Z}_n, 2)$ (see \cite[p 195]{AT}).
Denote the restriction map by $\eta : \operatorname{Aut}(E_{n+1})\to \operatorname{Aut}(\mathbb{K})$.
There is a map $\beta : K(\mathbb{Z}_n, 2)\to K(\mathbb{Z}, 3)=\operatorname{BAut}(\mathbb{K})$ that gives the Bockstein map $H^2(X, \mathbb{Z}_n)\to H^3(X)$ (see \cite{CS}).
By \cite[Lemma 3.20]{ST},
one has two elements $[\beta\circ P_2], [\operatorname{B}(\eta)]\in [\operatorname{BAut}(E_{n+1}), K(\mathbb{Z}, 3)]=H^3(\operatorname{BAut}(E_{n+1}))=\mathbb{Z}_n$.
Applying the Bockstein exact sequence and \cite[Theorem 3.15]{ST} for the reduced suspension of the Moore space $M_n$ (see \cite[Section 2]{r2}),
the above elements both generate $\mathbb{Z}_n$, and there is a number $r_0\geq 1, \operatorname{GCD}(n, r_0)=1$ with $r_0[\beta\circ P_2]=[\operatorname{B}(\eta)]\in\mathbb{Z}_n$.
For $X$ with ${\rm dim} X\leq 3$,
we have the following commutative diagram
$$\xymatrix{
[X, \operatorname{BAut}(E_{n+1}\otimes M_{(n)})]\ar[r]^{\operatorname{B}(\eta_1)_*}&\bar{E}^1_{M_{(n)}}(X)=H^3(X)\otimes \mathbb{Z}_{(n)}\\
[X, \operatorname{BAut}(E_{n+1})]\ar[d]^{{P_2}_*}\ar[u]^{\operatorname{B}(\otimes{\rm id}_{M_{(n)}})_*}\ar[r]^{\operatorname{B}(\eta)_*}&\operatorname{Tor}(H^3(X), \mathbb{Z}_n),\ar[u]\\
H^2(X, \mathbb{Z}_n)\ar@{->>}[ur]^{r_0\beta_*}&
}$$
and the left vertical maps are bijective by \cite[p 182]{WH},
where $\bar{E}^1_{M_{(n)}}(X) : =[X, \operatorname{BAut}_0(\mathbb{K}\otimes M_{(n)})]\subset E^1_{M_{(n)}}(X)$ is the reduced group (see \cite{DP2}).
Therefore we have $|\mathfrak{b}_{{(n)}}^{-1}(0)|=|H^2(X)\otimes \mathbb{Z}_n|$ by the Bockstein exact sequence
$$ 0\to H^2(X)\otimes \mathbb{Z}_n\to H^2(X, \mathbb{Z}_n)\xrightarrow{\beta_*} \operatorname{Tor}(H^3(X), \mathbb{Z}_n)\to 0.$$
\begin{rem}\label{dim}
By \cite[Theorem 1.2, p 112]{H},
two maps $H^2(X)\ni [L]\mapsto [L]-1\in \tilde{K}^0(X)$ and $H^2(X)\ni [L]\mapsto [L\oplus\underline{\mathbb{C}}^{\oplus n}]\in [X, \operatorname{B}U(n+1)]$ are bijective in the case of ${\rm dim}X\leq 3$.
The inverse map is given by the determinant map ${\rm det} : U(n+1)\to S^1$.
\end{rem}
\begin{lem}\label{o}
For a finite CW complex $X$ with ${\rm dim}X\leq 3$,
we have $([L]-1)([H]-1)=0\in \tilde{K}^0(X)$, $L, H\in\operatorname{Vect}_1(X)$.
\end{lem}
\begin{proof}
By Remark \ref{dim}, the map $\operatorname{B}({\rm det})_* : [X, \operatorname{B}U(2)]\to [X, \operatorname{B}S^1]$ is bijective.
Since $\operatorname{B}({\rm det})_*([L\oplus H])=[L\otimes H]=\operatorname{B}({\rm det})_*([(L\otimes H)\oplus \underline{\mathbb{C}}])$,
we have $[L][H]+1=[L]+[H]\in K^0(X)$.
\end{proof}
\begin{thm}\label{three}
For a finite connected CW complex $X$ with ${\rm dim} X\leq 3$,
the following holds $\colon$

1) $[X, \operatorname{BAut}(\mathcal{O}_{n+1})]=H^2(X, \mathbb{Z}_n)$,

2) $\operatorname{Im}\mathfrak{b}_{{(n)}}=\operatorname{Tor}(H^3(X), \mathbb{Z}_n)\subset E^1_{M_{(n)}}(X)$, $|\mathfrak{b}^{-1}_{(n)}(0)|=|H^2(X)\otimes \mathbb{Z}_n|$,

3) $\mathfrak{b}_{{(n)}}^{-1}(0)=O_1: =\{\mathcal{O}_E\, |\, E=L\oplus \underline{\mathbb{C}}^{\oplus n}, L\in \operatorname{Vect}_1(X)\}/\sim_{\rm isom}$.
\end{thm}
\begin{proof}
Since 1) and 2) are already proved,
we show 3).
By Lemma \ref{o}, we have $(n+([L]-1))[H]=n+([L\otimes H^{\otimes n}]-1) \in K^0(X)$ for $L, H\in \operatorname{Vect}_1(X)$, and Theorem \ref{dada} implies $$\mathcal{O}_{L\oplus \underline{\mathbb{C}}^{n}}\cong \mathcal{O}_{L'\oplus\underline{\mathbb{C}}^n}\Leftrightarrow [L]-[L']\in nH^2(X).$$
Thus we have $|\mathfrak{b}_{{(n)}}^{-1}(0)|=|H^2(X)\otimes \mathbb{Z}_n|=|O_1|<\infty$, and 3) is proved.
\end{proof}
\subsubsection{The case of the suspension $X=SY$}
Since the group $\operatorname{Aut}(\mathcal{O}_{n+1})$ is path connected,
we have $[SY, \operatorname{BAut}(\mathcal{O}_{n+1})]=[Y, \operatorname{Aut}(\mathcal{O}_{n+1})]=K_1(C(Y)\otimes \mathcal{O}_{n+1})$ where $SY$ is the non-reduced suspension of $Y$.
The invariant $\mathfrak{b}_{{(n)}}$ is given by ${\eta_n}_* : [Y, \operatorname{Aut}(E_{n+1}\otimes M_{(n)})]\to [Y, \operatorname{Aut}_0(\mathbb{K}\otimes M_{(n)})]\subset K^0(Y)\otimes \mathbb{Z}_{(n)}.$
The group structure of $[Y, \operatorname{Aut}(\mathcal{O}_{n+1})]$ was determined in \cite[Theorem 3.1]{r2},
and there is a group homomorphism $1-\delta : [Y, \operatorname{Aut}(\mathcal{O}_{n+1})]\to {K^0(Y)}^\times$.
The proof of Lemma \ref{triv} yields the following corollary.
\begin{cor}\label{sy}
The invariant $\mathfrak{b}_{{(n)}}$ is identified with the index map $\delta : K_1(C(Y)\otimes \mathcal{O}_{n+1}\otimes M_{(n)})\to \operatorname{Tor}(K^0(Y)\otimes \mathbb{Z}_{(n)}, \mathbb{Z}_n)$,
and for $[\alpha]\in [Y, \operatorname{Aut}(\mathcal{O}_{n+1})]$, we have $$\mathfrak{b}_{{(n)}}([\alpha])=1-\delta([u'\circ\alpha]_1)\in 1+\operatorname{Tor}(K^0(Y)\otimes \mathbb{Z}_{(n)}, \mathbb{Z}_n)\subset E^1_{M_{(n)}}(SY).$$
\end{cor}
Two rings $K^0(\mathbb{R}{\rm P}^5)=\mathbb{Z}[\nu]/(\nu^2+2\nu, \nu^3)$ and $K^0(\mathbb{R}{\rm P}^3)=\mathbb{Z}[\nu']/({\nu'}^2+2\nu', {\nu'}^2)$ are well-known (see \cite[Theorem 7.3]{JFA}).
By the following commutative diagram
$$\xymatrix{
&H^2(\mathbb{R}{\rm P}^5)\ar[r]\ar@{=}[d]&K^0(\mathbb{R}{\rm P}^5)\ar[d]\ar@{=}[r]&\mathbb{Z}\oplus \mathbb{Z}_4\\
\mathbb{Z}_2\ar@{=}[r]&H^2(\mathbb{R}{\rm P}^3)\ar[r]&K^0(\mathbb{R}{\rm P}^3)\ar@{=}[r]&\mathbb{Z}\oplus \mathbb{Z}_2
}$$
the generator of $H^2(\mathbb{R}{\rm P}^5)$ is sent to one of $1\pm{\nu}\in K^0(\mathbb{R}{\rm P}^5)$.
Since $1\pm{\nu}\not\in 1+\operatorname{Tor}(K^0(\mathbb{R}{\rm P}^5), \mathbb{Z}_2)$,
the image of $\mathfrak{b}_{{(n)}} : [S\mathbb{R}{\rm P}^5, \operatorname{BAut}(\mathcal{O}_{n+1})]\to E_{M_{(n)}}^1(S\mathbb{R}{\rm P}^5)$ dose not contain $H^3(S\mathbb{R}{\rm P}^5)$ by Corollary \ref{sy}.
However the invariant $\mathfrak{b}_{{(n)}}$ is not zero because $1+2\nu\in 1+\operatorname{Tor}(K_0(\mathbb{R}{\rm P}^5), \mathbb{Z}_2)\not=0$,
and this example suggests that the third cohomology alone is not enough to investigate the invariant.
\subsection{Questions}
We summarize some open problems.
\begin{prob}
Does the equation $|\mathfrak{b}_{{(n)}}^{-1}(z)|=|\mathfrak{b}_{{(n)}}^{-1}(0)|$ hold for non-trivial $z\not =0\in \operatorname{Im}\mathfrak{b}_{{(n)}}$?
Is there any operator algebraic realization, like Cuntz--Pimsner construction, of the elements in $\mathfrak{b}_{{(n)}}^{-1}(z)$?
\end{prob}
Since \cite{r2} shows that $\operatorname{BAut}(\mathcal{O}_{n+1})$ has no $H$-space structure,
it is not obvious that the image of $\mathfrak{b}_D$ is a subgroup of $E_D^1$ or not.
By \cite[Example 3.5]{DP2},
it might be worth pointing out the following questions.
\begin{prob}
Is the set $\operatorname{Im}\mathfrak{b}_D\subset E^1_D$ closed with respect to taking inverse?
Is the set $\operatorname{Im}\mathfrak{b}_D$ a subgroup of $E^1_D$?
\end{prob}
\begin{prob}
Is there a finite CW complex $X$ for which $\operatorname{Im}\mathfrak{b}_D$ contains a non-trivial $n^k$-torsion with $k>1$?
\end{prob}


\begin{thebibliography}{99}
\bibitem{JFA}J. F. Adams, Vector fields on spheres, Ann. of Math. (3) $\bf 75$ (1962), 603--632.
\bibitem{AP}C. A. Akemann, G. K. Pedersen and J. Tomiyama, Multipliers of C*-algebras, J. Functional analysis $\bf 13$, 277--301, 1973.
\bibitem{Bau}A. Bauval, $RKK(X)$-nucl\'earit\'e (d'apr\'es G. Skandalis), K-theory $\bf 13$ (1998), no. 1, 23--40.
\bibitem{Bl}B. Blackadar, K-theory for operator algebras, 2nd ed., Math. Sci. Inst. Publ., vol. $\bf 5$, Cambridge University Press, Cambridge, 1998.
\bibitem{C2}J. Cuntz, K-theory for cerain C*-algebras, Ann. of Math. (2) $\bf{113}$ : 1 (1981), 181--197.
\bibitem{EK}G. A. Elliott and D. Kucerovsky, An abstract Voiculescu--Brown--Douglas--Fillmore absorption theorem, Pacific J. Math. Vol. {\bf 198}, No.2, (2001), 385--409. 
\bibitem{DP}M. Dadarlat and U. Pennig, A Dixmier-Douady theory for strongly self-absorbing C*-algebras, J. Reine Angew. Math. $\bf{718}$ (2016), 153--181.
\bibitem{DP2}M. Dadarlat and U. Pennig, A Dixmier--Douady theory for strongly self-absorbing C*-algebras II : The Brauer group, J. Noncommut. Geom. 9 (2015), no. 4, 1137–-1154.
\bibitem{D1}M. Dadarlat, The C*-algebra of a vector bundle, J. Reine Angew. Math. $\bf{670}$ (2012), 121--143.
\bibitem{D2}M. Dadarlat, Continuous fields of C*-algebras over finite dimensional spaces, Adv. Math. 222 (2009), no. $\bf{5}$, 1850--1881.
\bibitem{D3}M. Dadarlat, Fiberwise $KK$-equivalence of continuous fields of C*-algebras, J. K-Theory 3 (2009), no. 2, 205--219.
\bibitem{DW}M. Dadarlat and W. Winter, On the KK-theory of strongly self-absorbing C*-algebras, Math. Scand. $\bf{104}$ (2009), no. 1, 95--107.
\bibitem{AT}J. F. Davis and P. Kirk, Lecture notes in algebraic topology, Graduate Studies in Mathematics, $\bf{35}$. Amer. Math. Soc, Providence, $\bf{RI}$, 2001.
\bibitem{DD}J. Dixmier and A. Douady, Champs continus d'espaces hilbertiens et de C*-algebres, Bull. Soc. Math. France $\bf{91}$ (1963),227--284.
\bibitem{JG}J. Gabe, A note on non-unital absorbing extensions, Pacific J. Math.,$\bf 284$(2), 383--393, 2016.
\bibitem{Hig}N. Higson and J. Cuntz, Kuiper's theorem for Hilbert modules, Contemporary Mathematics $\bf{62}$, 1987.
\bibitem{HRW}I. Hirshberg, M. R$\o$rdam and W. Winter, $C_0(X)$-algebras, stability and strongly self-absorbing C*-algebras, Mathematische Annalen 339 (3), 695-732.
\bibitem{H}D. Husemoller, Fibre bundles third edition., Grad. Texts Math. $\bf{20}$, Springer-Verlag, New York 1994.
\bibitem{r2}M. Izumi and T. Sogabe, The group structure of the homotopy set whose target is the automorphism group of the Cuntz algebra, International Jornal of Math. (11) $\bf 30$ (2019).
\bibitem{IM}M. Izumi and H. Matui, Poly-$\mathbb{Z}$ group actions on Kirchberg algebras II, arXiv:1906.03818v1.
\bibitem{XJ}X. Jiang, Nonstable K-theory for $\mathcal{Z}$-stable C*-algebras, arxiv:9707.5228[math.OA], 1997.
\bibitem{KK}G. G. Kasparov, Equivariant $KK$-theory and the Novikov conjecture. Invent. Math., 91(1) : 147--201, 1988.
\bibitem{KT}T. Katsura, On C*-algebras associated with C*-correspondences, Journal of Functional Analysis. $\bf{217}$ (2004), 366--401.
\bibitem{Lin}H. Lin, Full extensions and approximate unitary equivalence, Pacific J. Math. Vol. {\bf 229} No. 2, (2007), 389--428.
\bibitem{MN}R. Meyer and R. Nest, The Baum--Connes conjecture via localisation of categories, Topology $\bf 45$ (2006), n0. 2, 209--259.
\bibitem{M}J. A. Mingo, K-theory and multipliers of stable C*-algebras, Trans. Amer. Math. Soc. $\bf{299}$ : 1 (1987), 397--411.
\bibitem{P}W. L. Paschke and Salinas, Matrix algebras over $\mathcal{O}_{n+1}$, Michigan Math. J. $\bf{26}$ : 1 (1979), 3--12.
\bibitem{Phill}N. C. Phillips, A classification theorem for nuclear purely infinite simple C*-algebras, Doc. Math. $\bf{5}$ (2000), 49--114.
\bibitem{Pim}M. Pimsner, A class of C*-algebras generalizing both Cuntz-Krieger algebras and crossed products by $\mathbb{Z}$, Fields Inst. Commun. $\bf{12}$ (1997), 189--212.
\bibitem{RW}I. Raeburn and Dana P. Williams, Morita equivalence and continuous-trace C*-algebras, Mathematical Surveys and Monographs, $\bf{60}$. American Mathematical Society, Providence, $\bf{RI}$, 1998.
\bibitem{R2}M. R$\o$rdam, The stable and the real rank of $\mathcal{Z}$-absorbing C*-algebras, Internet. J. Math., $\bf 15$(10), 1065--1084, 2004.
\bibitem{CS}C. L. Schochet, The Dixmier--Douady invariant for Dummies, Notices Amer. Math. Soc. $\bf 56$ (2009), no. 7, 809--816.
\bibitem{ST}T. Sogabe, The homotopy groups of the automorphism groups of Cuntz-Toeplitz algebras, to appear in J. Math. Soc. Japan, arXiv:1903.02796.
\bibitem{TW}A. S. Toms and W. Winter, Strongly self-absorbing C*-algebras, Trans. Amer. Math. Soc. $\bf{359}$ (2007), no. 8, 3999--4029.
\bibitem{V}A. Valette, A remark on the Kasparov groups $\operatorname{Ext}(A, B)$, Pacific J. Math. $\bf{109}$ : 1 (1983), 247--255.
\bibitem{WH}G. W. Whitehead, Elements of Homotopy theory, Grad. Texts in Math. 61, Springer-Verlag, New York 1978.
\bibitem{W}W. Winter, Strongly self-absorbing C*-algebras are  $\mathcal{Z}$-stable, J. Noncommut. Geom. $\bf{5}$ (2011), no. 2, 253--264.













\end{thebibliography}
\end{document}